\documentclass[12pt, reqno, a4paper]{amsart}
\usepackage{amsfonts, amsmath, amsthm, amstext, amssymb}
\usepackage{mathtools}
\usepackage{url}
\usepackage{paralist}
\usepackage{tikz}\usetikzlibrary{arrows}
\usepackage{tikz-cd}
\usepackage[english]{babel}
\usepackage{framed}
\usepackage{amsrefs}
\usepackage[marginratio=1:1,height=660pt,width=410pt,tmargin=100pt]{geometry}
\usepackage{lipsum}
\usepackage{layout}
\usepackage{hyperref}
\usepackage{mathrsfs}

\DeclareFontFamily{U}{mathx}{\hyphenchar\font45}
\DeclareFontShape{U}{mathx}{m}{n}{
      <5> <6> <7> <8> <9> <10>
      <10.95> <12> <14.4> <17.28> <20.74> <24.88>
      mathx10
      }{}
\DeclareSymbolFont{mathx}{U}{mathx}{m}{n}
\DeclareFontSubstitution{U}{mathx}{m}{n}
\DeclareMathAccent{\widecheck}{0}{mathx}{"71}

\usepackage[OT2,T1]{fontenc}
\DeclareSymbolFont{cyrletters}{OT2}{wncyr}{m}{n}
\DeclareMathSymbol{\Sha}{\mathalpha}{cyrletters}{"58}
\DeclareMathAlphabet\mathbfcal{OMS}{cmsy}{b}{n}

\DeclareMathOperator{\tbP} {\bf{P}}

\DeclareMathOperator{\tbN} {\bf{N}}
\DeclareMathOperator{\tbW} {\bf{W}}
\DeclareMathOperator{\tbH} {\bf{H}}

\DeclareMathOperator{\tbU} {\bf{U}}
\DeclareMathOperator{\tbG} {\bf{G}}
\DeclareMathOperator{\tbL} {\bf{L}}
\DeclareMathOperator{\tbGL} {\bf{GL}}
\DeclareMathOperator{\mcMT} {\mathcal{MT}}
\DeclareMathOperator{\tbM} {\bf{M}}
\DeclareMathOperator{\hgt} {\bf{ht}}
\DeclareMathOperator{\Gr} {Gr}

\DeclareMathOperator{\CC} {\mathbb{C}}

\DeclareMathOperator{\RR} {\mathbb{R}}

\DeclareMathOperator{\QQ} {\mathbb{Q}}

\DeclareMathOperator{\ZZ} {\mathbb{Z}}
\DeclareMathOperator{\Lie} {Lie}
\DeclareMathOperator{\Ad} {Ad}
\DeclareMathOperator{\Zar} {Zar}
\DeclareMathOperator{\Gal} {Gal}
\DeclareMathOperator{\Mat} {Mat}
\DeclareMathOperator{\Hom} {Hom}

\newtheorem{lemma}{Lemma}[section]
\newtheorem{thm}[lemma]{Theorem}

\theoremstyle{definition}\newtheorem{definition}[lemma]{Definition}
\theoremstyle{definition}\newtheorem{remark}[lemma]{Remark}
\theoremstyle{definition}\newtheorem{notation}[lemma]{Notation}

\newcommand{\mb}[1]{\mathbb{#1}}
\newcommand{\mc}[1]{\mathcal{#1}}
\newcommand{\mf}[1]{\mathfrak{#1}}
\newcommand{\ra}{\rightarrow}
\newcommand{\ol}[1]{\overline{#1}}
\newcommand{\wt}[1]{\widetilde{#1}}
\newcommand{\tb}[1]{\textbf{#1}}
\newcommand{\bs}{\backslash}
\newcommand{\wc}{\widecheck}

\begin{document}
\title{Ax-Schanuel for variations of mixed Hodge structures}
\author{Kenneth Chung Tak Chiu}
\address{Department of Mathematics,  University of Toronto, Toronto, Canada.}
\email{kennethct.chiu@mail.utoronto.ca}
\begin{abstract}
We give properties of the real-split retraction of the mixed weak Mumford-Tate domain and prove the Ax-Schanuel property of period mappings arising from variations of mixed Hodge structures.  An ingredient in the proof is the definability of the mixed period mapping obtained by Bakker-Brunebarbe-Klingler-Tsimerman. In comparison with preceding results, in the point counting step, we count rational points on definable quotients instead.
\end{abstract}

\maketitle

\section{Introduction}

\subsection{Motivation}
In 1971, Ax proved the function field analogue \cite{Ax} of Schanuel's conjecture for exponentials.
This result was extended to other functions in variational Hodge theory, e.g. the $j$-function by Pila-Tsimerman  \cite{PT}, uniformizations of Shimura varieties by Mok-Pila-Tsimerman \cite{MPT}, variations of Hodge structures by Bakker-Tsimerman  \cite{BT}, and mixed Shimura varieties of Kuga type by Gao \cite{G}.
In this paper, we extend these results to variations of mixed Hodge structures. This was conjectured by Klingler \cite{K}.  
In another paper \cite{C}, the author used this result to prove the Ax-Schanuel theorem for derivatives  of mixed period mappings.
These generalizations of the Ax-Schanuel theorem are key results of functional transcendence, an area that has found fruitful applications in arithmetic geometry through the Pila-Zannier method \cite{PZ}. 
For example,  the recent work \cite{PSTEG} of Pila, Shankar and Tsimerman  on the Andr\'{e}-Oort conjecture uses the Ax-Lindemann-Weierstrass theorem (a specialization of the Ax-Schanuel theorem) for Shimura varieties established by Klingler-Ullmo-Yafaev \cite{KUY}, while the mixed case of the conjecture uses the corresponding theorem for  mixed Shimura varieties established by Gao \cite{G2}. 
Analogously, the Hodge theoretic generalizations of the Ax-Schanuel theorem mentioned above are used recently in several works \cite{BKU}\cite{BD}\cite{DR}\cite{PilSca} on the geometric aspects of the Zilber-Pink conjecture, which is a vast generalization of the Andr\'{e}-Oort conjecture.

Functional transcendence results have also been used in many other Diophantine problems. For example,  Lawrence-Venkatesh \cite{LV} and  Lawrence-Sawin  \cite{LS} used the Ax-Schanuel theorem for period mappings \cite{BT} to prove Shafarevich type conjectures for hypersurfaces. Gao used the Ax-Schanuel theorem for mixed Shimura varieties of Kuga type \cite{G} to study the generic rank of Betti map \cite{G4}, which was then used to prove a uniform bound for the number of rational points on curves by Dimitrov-Gao-Habegger \cite{DGH}.
Using the result in this paper, Hast develops a higher dimensional Chabauty-Kim method in \cite{H}.

\subsection{Statement of results}\label{Statement of results}
Notations from this section will be used throughout the paper including the appendix.
Let $X$ be a smooth quasiprojective irreducible algebraic variety over $\mb{C}$. Let 
$(\mc{H}, \mc{W}_\bullet, \mc{F}^\bullet, \mc{Q})$ be an admissible graded-polarized variation of mixed $\mb{Z}$-Hodge structures (GPVMHS) on $X$, where $\mc{W}_\bullet$ is the weight filtration, $\mc{F}^\bullet$ is the Hodge filtration, and $\mc{Q}$ is the graded polarization. Let $\eta$ be a Hodge generic point of $X$. Let $\Gamma$ be the image of the monodromy representation $\pi_1(X,\eta)\ra \tb{GL}(\mc{H}_{\mb{Z},\eta
})$ associated to the local system $\mc{H}$. Let $\textbf{P}$ be the identity component of the $\mb{Q}$-Zariski closure of $\Gamma$ in $\tb{GL}(\mc{H}_{\mb{Q}, \eta})$. Let $\textbf{U}$ be the unipotent radical of $\textbf{P}$.
Let $\tb{S}:=\text{Res}_{\mb{C}/\mb{R}} \mb{G}_m$ be the Deligne torus.
The Deligne splitting of the graded-polarized mixed Hodge structure $h_0$ on the stalk $\mc{H}_{\eta}$ defines a representation $\rho_0:\tb{S}_\mb{C}\ra \textbf{GL}(\mc{H}_{\mb{C},\eta})$ \cite[p. 7]{K}.
Let $\mc{M}$ be the space \cite[\S 3.5]{BBKT} parametrizing mixed $\mb{R}$-Hodge structures, with fixed graded-polarization and Hodge numbers which are the same as that of the mixed Hodge structures our GPVMHS is parametrizing. 
Let $\wc{\mc{M}}$ be the corresponding projective space parametrizing decreasing filtrations \cite[\S 3.5]{BBKT}.
Let $D$ be the $\textbf{P}(\mb{R})^+\textbf{U}(\mb{C})$-orbit of $h_0$ in $\mc{M}$, where $\tbP(\RR)^+$ is the identity component of $\tbP(\RR)$.
Let $\widecheck{D}$ be the $\tbP(\CC)$-orbit of $h_0$ in $\wc{\mc{M}}$.  
First assume $\Gamma\subset \tbP(\mb{Z})\cap \tbP(\mb{R})^+=:\tbP(\ZZ)^+$. This is assumed everywhere in the paper outside Theorem \ref{Ax-Schanuel for variations of mixed Hodge structures: mod gamma} and its proof.
Let $\psi: X\ra \Gamma\bs D$ be the period mapping\footnote{We have \emph{a priori} a period mapping $\wt{X}\ra D_{\mcMT}$, where $\wt{X}$ is the universal cover of $X$, and $D_{\mcMT}$ is the mixed Mumford-Tate domain, but it actually maps into $D$ (see Lemma \ref{weak period map}).}.   
Let $\varphi$ be the composition of $\psi$ with $\Gamma\bs D\ra \textbf{P}(\mb{Z})^+\bs D$. Consider the fiber product
\begin{center}
\begin{tikzcd}
W\arrow[r]\arrow[d] & D\arrow[d,"\pi"]
\\X\arrow[r, "\varphi"] & \textbf{P}(\mb{Z})^+\backslash D.
\end{tikzcd}
\end{center}

\begin{definition}
Let $h$ be any mixed $\mb{Z}$-Hodge structure in  $D$. Let $\textbf{M}$ be a normal algebraic $\mb{Q}$-subgroup of the Mumford-Tate group $\mcMT_{h}$ of $h$. Let $\tb{M}_u$ be its unipotent radical. Let $\tbM(\RR)^+$ be the identity component of $\tbM(\RR)$. The  $\textbf{M}(\mb{R})^+\tb{M}_u(\mb{C})$-orbit $D(\textbf{M})$ of $h$ is called a \textbf{weak Mumford-Tate domain}. For any $D(\textbf{M})\subset D$, any irreducible component of $\varphi^{-1}\pi (D(\textbf{M}))$ is called a \textbf{weakly special subvariety} of $X$.
\end{definition}

Let $p_X: X\times \widecheck{D} \ra X$ and $p_D: X\times \widecheck{D} \ra \widecheck{D}$ be the projections onto $X$ and $\widecheck{D}$ respectively. Let $U$ be an irreducible analytic subset of $W$, denote by $U^{\Zar}$ the Zariski closure of $U$ in $X\times \wc{D}$.
Let $V:=U^{\Zar}$. 
The main goal of this paper is to prove the following statement:

\begin{thm}\label{Ax-Schanuel for variations of mixed Hodge structures}
If
$\dim V-\dim U< \dim \wc{D},$
then $p_X(U)$ is contained in a proper weakly special subvariety.
\end{thm}

\begin{remark}\label{Remark, main statement}
This theorem is equivalent to the statement with $\varphi$ replaced by $\psi$, \emph{cf.} Lemma \ref{equivalent main statement}. By \cite[Corollary 6.7]{BBKT}, weakly special subvarieties are indeed algebraic. 
\end{remark}

Let $\psi': X\ra \Gamma\bs \mc{M}$ be the period mapping.
Let $p_X:X\times \mc{M}\ra X$ be the projection onto $X$.
Consider the fiber product
\begin{center}
\begin{tikzcd}
W'\arrow[r]\arrow[d] & \mc{M}\arrow[d,"\pi^\prime"]
\\X\arrow[r, "\psi'"] & \Gamma\bs \mc{M}.
\end{tikzcd}
\end{center}
Let $U'$ be an irreducible analytic subset of $W'$. Let $U'^{\Zar}$ be the Zariski closure of $U'$ in $X\times \wc{\mc{M}}$.   

\begin{thm}\label{Ax-Schanuel for variations of mixed Hodge structures: mod gamma}
If
$\dim U'^{\Zar}-\dim U'<\dim \wc{D},$
then $p_X(U')$ is contained in a proper weakly special subvariety.
\end{thm}

\subsection{Ideas of proof}\label{Idea of proof}
The main theorems will be proved by induction on 
$$(\dim X, \dim V-\dim U, \dim X- \dim U)$$
 in lexicographical order. 
 In Section \ref{Base cases of induction}, we will prove the base cases of induction.
Let $\tbN$ be the identity component of the $\mb{Q}$-Zariski closure of the $\tbP(\ZZ)^+$-stabilizer of $V$.
In the proofs of the Ax-Schanuel results in preceding works \cite{BT}\cite{G}\cite{MPT}, non-triviality of $\tbN$ was first obtained by applying the Pila-Wilkie counting theorem \cite{PW} on certain definable set $I$ and using the induction hypothesis. It was then used to construct a splitting of the period mappings. 
For this splitting to make sense, one has to prove that $\tbN$ is normal in the generic Mumford-Tate group, which could be strictly bigger than $\tbP$.
We use another approach instead: we apply the Pila-Wilkie theorem on the image $\ol{I}$ of $I$ under the map $\tbP(\mb{R})\ra (\tbP/\tbN)(\mb{R})$. 
The group $\tbN$ is indeed normal in $\tbP$, a fact which will be proved in Section \ref{Normality of the stabilizer} using the Hilbert scheme argument of Mok-Pila-Tsimerman \cite{MPT}. 

The definable set $I$ will be constructed as in the preceding works. It is defined in a way to facilitate the use of the induction hypothesis after the application of Pila-Wilkie on $\ol{I}$. At the first attempt, one perturbs $V$ and collects all  $\gamma\in \tbP(\RR)$ such that  $\dim \gamma V\cap W\geq \dim U$. However, since $I$ has to be definable, one modifies the attempt by further intersecting $W\cap \gamma V$ with $X\times \Phi$, where $\Phi$ is a definable fundamental set for the action of $\tbP(\mb{Z})^+$ on $D$. The definability of $W\cap (X\times \Phi)$ follows from the definability of the  mixed period mappings obtained recently by Bakker-Brunebarbe-Klingler-Tsimerman \cite{BBKT}. The precise construction of $I$ will be made in Section \ref{definable subset I of P(R)}. 

In order to use the counting theorem of Pila-Wilkie \cite{PW},  we need $\ol{I}$ to contain at least polynomially many rational points. We count rational points in $\ol{I}$ using the mixed point counting method  in \cite{G}. This method leads us to a trichotomy (Section \ref{Heights}), roughly as follows: 
\begin{enumerate}
\item The projection of $U$ to the reductive part, modulo the stabilizing part (since we are counting modulo $\tbN(\mb{Q}))$, has positive dimension, and the unipotent direction grows slower than the reductive direction. In this case, we apply the volume estimates of Griffith transverse subvarieties of a pure weak Mumford-Tate domain established by Bakker-Tsimerman \cite{BT}. This will be done in Section \ref{Proof of Case (1)}.
\item The unipotent direction grows faster than the reductive direction. In this case, we prove and apply a height estimate on products of certain conjuguates of upper unitriangular matrices (upper triangular matrix with 1's on the diagonal). This will be done in Section \ref{Proof of case (2)}.
\item The point count, modulo $\tbN(\mb{Q})$, is finite and  $U$ lies in a unipotent fiber. This case uses the definable Chow theorem \cite{PS}, see Section \ref{Proof of case (3)}.
\end{enumerate}

The trichotomy motivates the splitting of a domain, which parametrizes mixed Hodge structures, into three parts: the unipotent part, the stabilizing reductive part, and the non-stabilizing reductive part. 
 More precisely, in Section \ref{Real split locus of domains}, we first split the real-split retraction $r(D)$ of the  mixed weak Mumford-Tate domain $D$ into the unipotent part $D_{\tbU,\mb{R}}$ and the reductive part $D_{\text{Gr}}$ (Theorem \ref{splitting of retraction of weak mixed domain}).  Then we split the pure domain $D_{\text{Gr}}$ into two factors $D_{\tbN_r}$ and $D_{\tbL}$, where $\tbN_r$ is a Levi subgroup of $\tbN$, while $\tbL$ is defined in Section \ref{constructing definable fundamental set}. The first splitting makes use of the fact that $\tbU(\mb{R})$ and $\tbG(\mb{R})^+$ (where $\tbG$ is a Levi subgroup of $\tbP$ containing $\tbN_r$) act transitively on $D_{\tbU,\mb{R}}$ and $D_{\text{Gr}}$ respectively. These transitivities are proved in Lemma \ref{retraction of D is transitive} using Andr\'{e}'s normality \cite {A} of the connected algebraic monodromy group in the generic Mumford-Tate group, and results on the real-split retraction of the connected mixed Mumford-Tate domain $D_{\mcMT}^+$ by Bakker-Brunebarbe-Klingler-Tsimerman \cite[\textsection 6]{BBKT}.
 
 The definable fundamental domain $\Phi$ is built from the fundamental domains in each of $D_{\tbU,\mb{R}}$, $D_{\tbN_r}$ and $D_{\tbL}$, see  Section \ref{constructing definable fundamental set}.

In order to make sense of the comparison between the growths of the unipotent and the reductive directions in the trichotomy, we define the height of a subset of $r(D)$ in Section \ref{Heights}.

Preliminaries on the weak Mumford-Tate domains are collected in Appendix \ref{Appendix}.

Gao and Klingler \cite{G3} prove the same results independently using a similar approach at the same time. 
The main difference is that we count points in the group $(\tbP/\tbN)(\CC)$ directly as explained earlier. The point count results in Sections \ref{Proof of Case (1)} to \ref{Proof of case (3)} of our paper are applied in the proof of the geometric Andr\'{e}-Grothendieck period conjecture by Bakker-Tsimerman \cite{BT22} recently.

\subsection{Acknowledgements}
I would like to thank my advisor Jacob Tsimerman for suggesting me to work on this problem, and for many insightful discussions and helpful comments.

\section{Base cases of induction}\label{Base cases of induction}
We prove Theorem \ref{Ax-Schanuel for variations of mixed Hodge structures} and Theorem \ref{Ax-Schanuel for variations of mixed Hodge structures: mod gamma} simultaneously by induction on $\dim X$. 
The case when $\dim X=0$ is trivial. Suppose $\dim X>0$.
For each $\dim X$, we prove Theorem \ref{Ax-Schanuel for variations of mixed Hodge structures}  by induction on 
$(\dim V-\dim U, \dim X- \dim U)$
in lexicographical order, 
and deduce Theorem \ref{Ax-Schanuel for variations of mixed Hodge structures: mod gamma} from  Theorem \ref{Ax-Schanuel for variations of mixed Hodge structures}.

If the monodromy group $\Gamma$ is contained in $\tbP(\ZZ)^+$, let $W_{\Gamma}:= X\times_{\Gamma\bs D} D$, and let $S$ be the set of all distinct representatives of the cosets in $\tbP(\ZZ)^+/\Gamma$, we have $W=\bigcup_{g\in S}gW_{\Gamma}$.
Then since $U$ is irreducible, $g_0^{-1}U\subset W_{\Gamma}$ for some $g_0\in S$.
The following lemma then follows:

\begin{lemma}\label{equivalent main statement}
Theorem \ref{Ax-Schanuel for variations of mixed Hodge structures} is equivalent to the following: Assume $\Gamma\subset \tbP(\ZZ)^+$. Let $U_{\Gamma}:=g_0^{-1}U$. Let $V_{\Gamma}:=U_{\Gamma}^{\Zar}$. 
If
$\dim V_{\Gamma}-\dim U_{\Gamma}< \dim \wc{D},$
then $p_X(U_{\Gamma})$ is contained in a proper weakly special subvariety.
\end{lemma}

\begin{lemma}\label{equivalent main statement 2}
Let $k$ be an integer. If Theorem \ref{Ax-Schanuel for variations of mixed Hodge structures} holds for $\dim X=k$, then Theorem \ref{Ax-Schanuel for variations of mixed Hodge structures: mod gamma} also holds for $\dim X=k$.
\end{lemma}

\begin{proof}
Since $\Gamma\cap \tbP(\QQ)$ is of finite index in $\Gamma$, replacing $X$ by a finite \'{e}tale covering if necessary, we can assume that $\Gamma\subset \tbP(\mb{Z})$. Similarly, we can further assume that $\Gamma\subset \tbP(\mb{R})^+$. 
Then since the image of $\psi'$ is inside $\Gamma\bs D$, $W_{\Gamma}=W'$, and so Theorem \ref{Ax-Schanuel for variations of mixed Hodge structures: mod gamma} follows from Lemma \ref{equivalent main statement}.
\end{proof}

\begin{lemma}\label{Induction on dimension of X}
 If there exists an algebraic subvariety $Z$ of $X$ such that $p_X(U)\subset Z\subsetneq X$, then Theorem \ref{Ax-Schanuel for variations of mixed Hodge structures} holds.
 \end{lemma}

\begin{proof}
By Lemma \ref{equivalent main statement}, it suffices to prove the statement about $U_{\Gamma}$ in the lemma.
Let  $\Gamma_{Z^{sm}}$ be the monodromy group of the GPVMHS restricted to the smooth locus $Z^{sm}$ of $Z$. 
We have a mapping $Z^{sm}\ra \Gamma_{Z^{sm}}\bs \mc{M}$.
Let $$W_{Z^{sm}}:= {Z^{sm}}\times_{\Gamma_{Z^{sm}}\bs \mc{M}}\mc{M}.$$ 
Since $\dim {Z^{sm}}<\dim X$, by induction hypothesis, Theorem \ref{Ax-Schanuel for variations of mixed Hodge structures: mod gamma} holds for ${Z^{sm}}\ra \Gamma_{Z^{sm}}\bs \mc{M}$.
Let $W'_{Z^{sm}}:= {Z^{sm}}\times_{\Gamma\bs \mc{M}}\mc{M}$. 
Let $S'$ be the set of all distinct representatives of the cosets in $\Gamma/\Gamma_{Z^{sm}}$. 
We have $W'_{Z^{sm}}=\bigcup_{g\in S'}gW_{Z^{sm}}$.
For any irreducible analytic subset $U'_{Z^{sm}}$ of $W'_{Z^{sm}}$, we have $g^{-1}U'_{Z^{sm}}\subset W_{Z^{sm}}$ for some $g\in S'$.
Hence, Theorem \ref{Ax-Schanuel for variations of mixed Hodge structures: mod gamma} holds for ${Z^{sm}}\ra \Gamma\bs \mc{M}$.
Since $U_{\Gamma}:=g_0^{-1}U\subset W_{\Gamma}$,  we know  $U_{\Gamma}\cap p_X^{-1}(Z^{sm})\subset W'_{Z^{sm}}$. 
We can assume $p_X(U)\cap Z^{sm}$ is non-empty.
We also have $p_X(g_0^{-1}U)=p_X(U)\subset Z$.
Then $p_X(U_{\Gamma})\cap Z^{sm}$ is open and dense in $p_X(U)$,
while $U_{\Gamma}\cap p_X^{-1}(Z^{sm})$ is open and dense in $U_{\Gamma}$.
If the identity component $\ol{\Gamma}^\circ_{Z^{sm}}$ of the algebraic monodromy group of $Z^{sm}$ is equal to $\tbP$, then $\dim \wc{D}_{Z^{sm}}=\dim \wc{D}$.
Then since Theorem \ref{Ax-Schanuel for variations of mixed Hodge structures: mod gamma} holds for ${Z^{sm}}\ra \Gamma\bs \mc{M}$, 
the set $p_X(U_{\Gamma})\cap Z^{sm}$ is contained in a proper weakly special subvariety $E_1$.
If $\ol{\Gamma}_{Z^{sm}}^\circ\subsetneq \tbP$, then by Lemma \ref{Andre-Deligne: monodromy and proper weakly special},
$Z^{sm}$ is contained in a proper weakly special subvariety $E_2$. 
The Zariski closures of $p_X(U_{\Gamma})$ and $p_X(U_{\Gamma})\cap Z^{sm}$ in $X$ are the same, and they are contained in $E_1$ or $E_2$ by Lemma \ref{monodromy and smallest weakly special}.
Therefore, the statement about $U_{\Gamma}$ in Lemma  \ref{equivalent main statement} holds.
 \end{proof}

\begin{lemma}\label{algebraic projection}
If $p_X(U)$ is contained in an algebraic subvariety $Z$ and contains an analytic open subset of $Z$, then Theorem \ref{Ax-Schanuel for variations of mixed Hodge structures} holds.
\end{lemma}

\begin{proof}
Suppose $p_X(U)$ contains an analytic open subset of $X$. 
Then $p_X(g_0^{-1}U)$ contains an analytic open subset of $X$.
Since $g_0^{-1}U\subset W_{\Gamma}$ and $\Gamma$ acts discretely on $D$, we have
$\dim g_0^{-1}U=\dim W_{\Gamma}.$
Since $W_{\Gamma}$ is irreducible, we then have $W^{\Zar}_{\Gamma}=g_0^{-1}U^{\Zar}$. 
Since $\Gamma$ is $\QQ$-Zariski dense in $\tbP$, we know   $\Gamma$ is $\mb{C}$-Zariski dense in $\tb{P}_{\mb{C}}$ by Lemma \ref{Galois descent, Zariski closure}. 
Then since $W_{\Gamma}$ is invariant under $\Gamma$,  we know $U^{\Zar}$ is invariant under $\tb{P}_{\mb{C}}$.  Moreover, $p_X(U^{\Zar})$ contains an analytic open subset of $X$. Then since $\tb{P}_{\mb{C}}$ acts transitively on $\widecheck{D}$, we know $V:=U^{\Zar}$ contains $X^{\circ}\times \widecheck{D}$, where $X^\circ$ is an open subset of $X$. Then by assumption, 
 $\dim U>\dim V-\dim \wc{D}=\dim X=\dim W$, which is a contradiction. Therefore, $p_X(U)$ does not contain any analytic open subset of $X$. If $p_X(U)$ is contained in an algebraic subvariety $Z$ and contains an analytic open subset of $Z$, then $Z\subsetneq X$. By Lemma \ref{Induction on dimension of X}, Theorem \ref{Ax-Schanuel for variations of mixed Hodge structures} holds .
\end{proof}

Since $U\subset W$, the fibers of $(p_X)|_U$ are discrete sets, so if $\dim V=\dim U$, then $\dim p_X(V)=\dim p_X(U)$, and thus $p_X(U)$ contains an analytic open subset of $p_X(V)$. Then since $p_X(V)$ is algebraic, Theorem \ref{Ax-Schanuel for variations of mixed Hodge structures} holds by Lemma \ref{algebraic projection}.

Similarly, if $\dim X=\dim U$, then $\dim X=\dim U=\dim p_X(U)$,  Theorem \ref{Ax-Schanuel for variations of mixed Hodge structures} thus holds by Lemma \ref{algebraic projection}.

\section{Retraction of the mixed weak Mumford-Tate domain}\label{Real split locus of domains}
The motivation of this section is explained in the idea of proof (Section \ref{Idea of proof}).  
A subset of $\RR^n$ is said to be \emph{definable} if it is definable in the o-minimal structure $\RR_{an,\exp}$ \cite{DM}. We refer to \cite[Section 2]{JW} for an introduction to o-minimality.

The Mumford-Tate groups of mixed Hodge structures in the variation attain maximal dimension and are isomorphic to each other for all $x$ outside a countable union (called the Hodge locus) of proper irreducible analytic varieties of $X$ \cite[Lemma 4]{A}. We say these Mumford-Tate groups are \emph{Hodge generic}. If $x$ is in the Hodge locus, then the Mumford-Tate group is strictly smaller than the generic one. Points outside the Hodge locus are also called \emph{Hodge generic points}.

Let $\mc{S}(W_{\bullet})$ be the variety of splittings of the weight filtration of $h_0$ \cite{BBKT}\cite{CKS}, where $h_0$ is the mixed Hodge structure at the Hodge generic point $\eta$ chosen in Section \ref{Statement of results}.
Let $\mcMT_u$ be the unipotent radical of the Mumford-Tate group $\mcMT:=\mcMT_{h_0}$ of $h_0$. 
Let $D^+_{\mcMT}$ be the $\mcMT(\mb{R})^+\mcMT_u(\mb{C})$-orbit of $h_0$ in the space $\mc{M}$ defined in Section \ref{Statement of results}. 
By taking $\tbM=\mcMT_u$ in Appendix \ref{Appendix}, we obtain a mixed Hodge structure $\ol{h_0}$ (\emph{cf.} Lemma \ref{quotient of mixed hodge structure}) and a pure Mumford-Tate domain $D_{\mcMT/\mcMT_u}$ with an $(\mcMT/\mcMT_u)(\RR)$-action on it. We write $D_{\mcMT, \text{Gr}}:=D_{\mcMT/\mcMT_u}$ and call it the Mumford-Tate domain for the associated graded $\ol{h_0}$.  Let $D_{\mcMT,\Gr}^+$ be the connected component of $D_{\mcMT,\Gr}$ containing $\ol{h_0}$.

Let $D_{\mcMT, \mb{R}}^+$ be the real semi-algebraic subset of $D_{\mcMT}^+$ consisting of those Hodge filtrations for which the corresponding mixed Hodge structure is real-split. There is an $\mcMT(\mb{R})^+$-equivariant retraction $r:  D_{\mcMT}^+\ra D_{\mcMT, \mb{R}}^+$ (\cite[Lemma 6.6]{BBKT}). Let $D_{\mcMT_u,\mb{R}}$ be the $\mcMT(\mb{R})^+$-orbit in $\mc{S}(W_{\bullet})(\mb{R})$ of the splitting of $h_1:=r(h_0)$. The unipotent radical $\mcMT_u(\mb{R})$ acts simply transitively on $D_{\mcMT_u,\mb{R}}$ \cite[Prop. 2.2]{CKS}, see also \cite[p. 13]{BBKT} or \cite[Lemma 3.11]{C1}. Let $h_{\text{Gr}}$ denote the image of any $h\in D_{\mcMT}^+$ under the map $D_{\mcMT}^+\ra D_{\mcMT,\text{Gr}}^+$ (\emph{cf.} Appendix \ref{Appendix}).  

\begin{notation}\label{Levi decomposition of the retraction}
Recall $\mc{M}$ in Section \ref{Statement of results}. Then $D_{\mcMT}^+\subset \mc{M}$. Also, $D^+_{\mcMT,\mb{R}}$ is a subset of the real-split locus $\mc{M}_{\mb{R}}$ of $\mc{M}$. There is a product $\Omega=\prod \Omega_k$, where each $\Omega_k$ parametrizes pure Hodge structures having the $k$-th polarization and Hodge numbers of weight $k$ that we fixed. For each $h\in \mc{M}_{\mb{R}}$, there exists a unique Deligne bigrading \cite{D}, which gives a splitting $h_s\in S(W_{\bullet})(\mb{R})$. Let $h_{\text{Gr}}$ be the associated graded of $h$.  Then we have a bijection $\mc{M}_{\mb{R}}\simeq \Omega\times S(W_{\bullet})(\mb{R})$ sending $h$ to $(h_{\text{Gr}},h_s)$. Let  $i: D_{\mcMT,\mb{R}}^+\simeq D_{\mcMT,\text{Gr}}^+\times D_{\mcMT_u,\mb{R}}$ be the restriction of this bijection to $D_{\mcMT,\mb{R}}^+$. The map $i$ is definable, and it is compatible with the $\mcMT(\mb{R})^+$-action, i.e. $i(m h)=(\ol{m} h_{\text{Gr}}, mh_s)$ for any $m\in \mcMT(\mb{R})^+$ and $h\in D_{\mcMT,\mb{R}}^+$, where $\ol{m}$ is the image of $m$ in $(\mcMT/\mcMT_u)(\RR)$. 
For any $h\in D_{\mcMT,\mb{R}}^+$, we also denote by $(h_\text{Gr}, h_s)$ the image of $h$ under $i$.
$\hfill\square$
\end{notation}

Recall $D$ and $\tbP$ from Section \ref{Statement of results}.
Let $D_{\text{Gr}}$ be the image of $r(D)$ under the  $\mcMT(\mb{R})^+$-equivariant  map $D_{\mcMT,\mb{R}}^+\ra D_{\mcMT,\text{Gr}}^+$.  By the equivariance of the map $D_{\mcMT,\mb{R}}^+\ra D_{\mcMT,\text{Gr}}^+$, the $\tbP(\mb{R})^+$-action on $r(D)$ induces a  $\tbP(\mb{R})^+$-action on $D_{\text{Gr}}$. Let $D_{\tbU,\mb{R}}$ be the image of $r(D)$ under the projection $D_{\mcMT,\mb{R}}^+\ra D_{\mcMT_u,\mb{R}}$. Similarly, we have a $\tbP(\mb{R})^+$-action on $D_{\tbU,\mb{R}}$.

 Let $\tbN$ be the identity component of the $\mb{Q}$-Zariski closure of 
 $$Stab(V):=\{\sigma\in \tb{P}(\mb{Z})^+: \sigma V=V\}$$
  in $\tb{P}$, where $\tbP(\ZZ)^+:=\tbP(\mb{Z})\cap \tbP(\mb{R})^+$. 
 Fix a Levi subgroup $\tbN_r$ of $\tbN$. Let $\tbG:=\tbP_r$ be a maximal connected reductive subgroup  of $\tbP$ containing $\tbN_r$. Then $\tbG$ is a Levi subgroup of $\tbP$. Similarly, we can choose Levi subgroup $\mcMT_r$ of $\mcMT$ containing $\tbG$. 
We have an isomorphism $\mcMT_r\simeq \mcMT/\mcMT_u$, given by inclusion $\mc{MT}_r\ra \mc{MT}$ followed by the quotient map $\mcMT\ra \mcMT/\mcMT_u$.
Let $\tbP_u$ be the unipotent radical of $\tbP$.

\begin{lemma}\label{retraction of D is transitive}
The group $\tbG(\mb{R})^+:=\tbP_r(\mb{R})^+$ acts transitively on $D_{\Gr}$. The group $\tbU(\mb{R}):=\tbP_u(\mb{R})$ acts simply transitively on $D_{\tbU,\mb{R}}$. 
\end{lemma}

\begin{proof}
By Andr\'{e} \cite[Proof of Theorem 1]{A}, $\tbP$ is normal in $\mcMT$. 
Taking $\tbM=\tbP$ in Appendix \ref{Appendix}, we obtain a quotient morphism 
$$f: (\mcMT, \mc{X}_{\mcMT}, D_{\mcMT}^+)\ra (\mcMT/\tbP, \mc{X}_{\mcMT/\tbP}, D_{\mcMT/\tbP}^+)$$
between connected mixed Hodge data.   
Let $D_{\mcMT/\tbP,\mb{R}}^+$ be the real split locus of $D_{\mcMT/\tbP}^+$. By \cite[Lemma 6.6]{BBKT}, we have an $(\mcMT/\tbP)(\mb{R})^+$-equivariant retraction $r_{\mcMT/\tbP}: D_{\mcMT/\tbP}^+\ra D_{\mcMT/\tbP,\mb{R}}^+$, satisfying $r_{\mcMT/\tbP}\circ f=f_{\mb{R}}\circ r$, where $f_{\mb{R}}: D_{\mcMT,\mb{R}}^+\ra D_{\mcMT/\tbP,\mb{R}}^+$ is the morphism induced from $f$ by restriction. The weak Mumford-Tate domain $D$ is in the fiber $f^{-1}(x)$ of a point $x\in D_{\mcMT/\tbP}^+$. Then $f_{\mb{R}}(r(D))=\{r_{\mcMT/\tbP}(x)\}$. 

By \cite[Corollary 14.11]{Bor}, $\mcMT_u\tbP/\tbP$ is the unipotent radical of $\mcMT/\tbP$. By \cite[Prop. 2.13]{G2}, $\tbP_r=\tbP\cap  \mcMT_r$. Then $\mcMT_r/\tbP_r$ is isomorphic to $\mcMT_r\tbP/\tbP$, a Levi subgroup of $\mcMT/\tbP$.

 Let $g_1: D_{\mcMT,\text{Gr}}^+\ra  D_{\mcMT/\tbP,\text{Gr}}^+$ be the quotient morphism obtained by pushing forward representations of the Deligne torus via $\mcMT_r\ra \mcMT_r/\tbP_r$, \emph{cf.} Appendix \ref{Appendix}. 
 By the previous paragraph about Levi subgroups,  we have the following commutative diagram:
\begin{center}
\begin{tikzcd}
r(D)\arrow[r, hook]\arrow[d] & D_{\mcMT, \mb{R}}^+ \arrow[d]\arrow[r, "f_{\mb{R}}"] & D_{\mcMT/\tbP, \mb{R}}^+\arrow[d]
\\D_{\text{Gr}}\arrow[r, hook] &  D_{\mcMT,\text{Gr}}^+\arrow[r, "g_1"]& D_{\mcMT/\tbP,\text{Gr}}^+.
\end{tikzcd}
\end{center}
The conclusion in the first paragraph that $f_{\mb{R}}(r(D))=\{r_{\mcMT/\tbP}(x)\}$ then implies that $g_1$ maps $D_{\text{Gr}}$ to a point $y$, which is the associated graded of $r_{\mcMT/\tbP}(x)$. 
By Lemma \ref{D and fiber of quotient}, the group $\tbP_r(\mb{R})^+$ acts transitively on $D_{\text{Gr}}$.
 
 By Notation \ref{Levi decomposition of the retraction}, we have an isomorphism $D_{\mcMT/\tbP,\mb{R}}^+\simeq D_{\mcMT/\tbP,\text{Gr}}^+\times D_{(\mcMT/\tbP)_u,\mb{R}}$, where $D_{(\mcMT/\tbP)_u,\mb{R}}$ is the $(\mcMT/\tbP)_u(\mb{R})$-orbit of the splitting of $f_{\mb{R}}(h_1)$. Write as $(h_{\text{Gr}}, h_s)$ the image of any $h\in  D_{\mcMT/\tbP,\mb{R}}^+$ under this isomorphism. Define a map $g_2: D_{\mcMT_u,\mb{R}}\ra D_{(\mcMT/\tbP)_u,\mb{R}}$ by $g_2(u\cdot {h_1}_s)= f_{\mb{R}}(u\cdot h_1)_s$.
We have the following commutative diagram
\begin{center}
\begin{tikzcd}
r(D)\arrow[r, hook]\arrow[d] & D_{\mcMT, \mb{R}}^+ \arrow[d]\arrow[r, "f_{\mb{R}}"] & D_{\mcMT/\tbP, \mb{R}}^+\arrow[d]
\\D_{\tbU,\mb{R}}\arrow[r, hook] &  D_{\mcMT_u,\mb{R}}\arrow[r, "g_2"]& D_{(\mcMT/\tbP)_u,\mb{R}}.
\end{tikzcd}
\end{center}
Since  $f_{\mb{R}}(r(D))=\{r_{\mcMT/\tbP}(x)\}$, we know $g_2$ maps $D_{\tbU,\mb{R}}$ to a point. Then since $(\mcMT/\tbP)_u(\mb{R})$ acts simply transitively on $D_{(\mcMT/\tbP)_u,\mb{R}}$ \cite[Prop. 2.2]{CKS} (see also \cite[p. 13]{BBKT} or \cite[Lemma 3.11]{C1}) and since  $\tbP_u=\tbP\cap  \mcMT_u$ \cite[Prop. 2.13]{G2}, the group $\tbU(\mb{R}):=\tbP_u(\mb{R})$ acts transitively on $D_{\tbU,\mb{R}}$. This action is simply transitive because $\mcMT_u(\mb{R})$ acts simply transitively on $D_{\mcMT_u,\mb{R}}$.
\end{proof}

\begin{thm}\label{splitting of retraction of weak mixed domain}
The group $\tbP(\mb{R})^+$ acts transitively on $r(D)$.
The isomorphism $i$ restricts to a $\tbP(\mb{R})^+$-equivariant definable isomorphism $j: r(D)\simeq D_{\Gr}\times D_{\tbU,\mb{R}}$.
\end{thm}

\begin{proof}
It is clear that $j$ is $\tbP(\RR)^+$-equivariant.
Let  $(x, y)\in D_{\text{Gr}}\times D_{\tbU,\mb{R}}$. By Lemma \ref{retraction of D is transitive}, $x= g {h_1}_{\text{Gr}}$ and $y=u{h_1}_s$ for some $g\in \tbG(\mb{R})^+$ and $u\in \tbU(\mb{R})$. There exists $u^\prime \in \tbU(\mb{R})$ such that $g^{-1}{h_1}_s=u^\prime {h_1}_s$. Then $j(ugu^\prime h_1)=(ugu^\prime {h_1}_{\text{Gr}}, ug u^\prime {h_1}_s)=(g{h_1}_{\text{Gr}}, u{h_1}_s)=(x,y)$ since the unipotent radical acts trivially on the associated graded. Therefore, $j$ is surjective, and the group $\tbP(\RR)^+$ acts transitively on $r(D)$. Since $i$ and $r(D)$ are definable, $j$ is definable.
\end{proof}

\section{Normality of algebraic stabilizer in algebraic monodromy group}\label{Normality of the stabilizer}

\subsection{A temporary definable fundamental set $\Phi'$ for $\tb{P}(\mb{Z})^+\bs D$}\label{temporary definable fundamental set}

The definable fundamental set for $\tbP(\mb{Z})^+\bs D$ in this section is temporary because later on when we define the definable set $I$ in Section \ref{definable subset I of P(R)}, we will switch to another fundamental set in Section \ref{constructing definable fundamental set} that depends on the normality we prove in this section. 

\begin{definition}[\cite{BBKT}]
Let $Y$ be a definable locally compact subset in $\RR^n$ and $\Gamma$ a group acting on $Y$ by definable homeomorphism. A subset $F$ in $Y$ is a \textbf{fundamental set} for $\Gamma\bs Y$ if $\Gamma\cdot F=Y$ and the set $\{g\in \Gamma: F\cap gF\neq \varnothing\}$ is finite.
\end{definition}

The intersection $\Lambda:=\tb{P}(\mb{Z})^+\cap\tbG(\mb{Q})$ is an arithmetic subgroup containing $\tb{G}(\mb{Z})^+:=\tbG(\ZZ)\cap \tbG(\RR)^+$. Let $\Phi_{\tb{G}}$ be a definable open fundamental set for the action of $\Lambda$ on $D_{\text{Gr}}$ \cite[Theorem 1.1]{BKT}. Let $\Phi_{\tb{U}}$ be a bounded definable open fundamental set for the cocompact action of $\tb{U}(\mb{Z})$ on  $D_{\tb{U},\mb{R}}$ \cite[Lemma 4.7 (1)]{PR}.

\begin{lemma}\label{product of fundamental sets is a fundamental set}
The set $\Phi_{\mb{R}}^\prime:=j^{-1}(\Phi_{\tbG}\times \Phi_{\tbU})$ is a definable open fundamental set for the action of $\tbP(\mb{Z})^+$ on $r(D)$. Hence $\Phi^\prime:=r^{-1}(\Phi_{\mb{R}}^\prime)$ is a definable open fundamental set for the action of $\tbP(\mb{Z})^+$ on $D$.
\end{lemma}

\begin{proof}
This follows from \cite[Lemma 4.7 (2)]{PR}.
\end{proof}

\subsection{Normality of algebraic stabilizer in algebraic monodromy group}
Recall that $\tbN$ is the identity component of the $\mb{Q}$-Zariski closure of $Stab(V):=\{\sigma\in \tb{P}(\mb{Z})^+: \sigma V=V\}$ in $\tb{P}$. We apply the Hilbert scheme argument to prove that $\tbN$ is normal in $\tb{P}$. This argument was used in  \cite{BT}, \cite{G}, and  \cite{MPT}. 

Let $(X\times \wc{D})'$ be a projective compactification of $X\times \wc{D}$.
Let $M$ be the Hilbert scheme of all subvarieties of $(X\times \widecheck{D})'$ with the same Hilbert polynomial as $V'$, where $V'$ is the Zariski closure of $V$ in $(X\times \wc{D})'$. Let $\mc{V}\ra M$ be the universal family over $M$, with a natural embedding $\mc{V}\hookrightarrow (X\times \widecheck{D})'\times M$. Let $\mc{V}_W:=\mc{V}\cap (W\times M)$. Each $m\in M$ corresponds to a subvariety called $V_m$. Write $m=[V_m]$.

Let $T$ be the set of all pairs $(p,m)\in W\times M$, such that  $V_m\cap W$ has dimension at least $\dim U$ around $p$. 
The set $T$ is closed and analytic in $\mc{V}_W$, see proof of \cite[Lemma 8.2]{PS08}.
Let $T_0$ be the irreducible component containing $(p,[V])$ for some and hence any $p\in U$.

The action of $\tbP(\mb{Z})^+$ on $X\times D$, defined by $\gamma\cdot (x,h)=(x,\gamma\cdot h)$, lifts to $\mc{V}_W$. There is also an action of $\tbP(\mb{Z})^+$ on $T$. Let $Y:=\tbP(\mb{Z})^+\bs T_0$ be the image of $T_0$ in $\tbP(\mb{Z})^+\bs T$.

\begin{lemma}\label{definability varphi}
The period mapping $\varphi:X\ra \tbP(\mb{Z})^+\setminus D$ is definable.
\end{lemma}

\begin{proof}
We follow Bakker-Brunebarbe-Klingler-Tsimerman \cite[\textsection 5]{BBKT} and make suitable modifications (recall that $\tbP$ is the connected algebraic monodromy group). By \cite[Lemma 4.1]{BBKT}, by passing $X$ to a finite \'{e}tale covering if necessary,  $X$ is the union of finitely many punctured polydisks such that the GPVMHS has unipotent monodromy over each such polydisk. It suffices to prove that the period map $\varphi|_{(\Delta^*)^n}$ restricted to each such polydisk, say $(\Delta^*)^n$, is definable. By \cite[Prop. 5.2]{BBKT}, the restriction  to any vertical strip $E$ of the lifting $\wt{\varphi}$ of $\varphi|_{(\Delta^*)^n}$  is definable, so it suffices to prove that the image $\wt{\varphi}(E)$  lies in a finite union of definable fundamental sets of $\tbP(\mb{Z})^+\setminus D$. By \cite[Cor. 2.34]{BP}, the composition of $\wt{\varphi}$ with $D\ra D_{\tbU,\mb{R}}$ is bounded on any vertical strip. 
It suffices to prove that the image of $\wt{\varphi}(E)$ in $D_{\Gr}$ lies in a finite union of Siegel sets.
By \cite[Cor. 2]{A}, $\tbG$ is semisimple. Then by \cite[Theorem 1.2]{Orr} and \cite[Prop. 3.4]{BGST}, it suffices to prove that $\wt{\varphi}(E)$ lies in a finite union of  Siegel sets in $D^+_{\mcMT,\Gr}$.
This holds by \cite[Theorem 1.5]{BKT}.
\end{proof}

Hence, the set 
$$W\cap (X\times \Phi')=\{(x,F^\bullet)\in X\times \Phi': \varphi(x)=\pi|_{\Phi'} (F^\bullet)\}$$
is definable.

Since the Hilbert scheme $M$ is proper, the composition
$T\hookrightarrow W\times M\ra W$
 is proper, so $\tbP(\mb{Z})^+\bs T\ra \tbP(\ZZ)^+\bs W=X$ is proper, and thus the induced map $q: Y\ra X$ is proper. The intersection $\mc{V}\cap((W\cap(X\times \Phi^\prime))\times M)$  is a definable fundamental set for the action of $\tbP(\mb{Z})^+$ on $\mc{V}_{W}$. By  \cite[Proposition 2.3]{BBKT}, $\tbP(\mb{Z})^+\bs \mc{V}_W$ and similarly $Y$ have definable structures, so the projection $q$ is definable. Then $q(Y)$ is closed, complex analytic and definable in $X$, and therefore algebraic by definable Chow \cite{PS}. 
 
Since $q(Y)\supset p_X(U)$, by Lemma \ref{Induction on dimension of X}, we can assume $q(Y)=X$. Let $\mc{F}$ be the family of algebraic varieties parametrized by the projection of $T_0$ in $M$. The family $\mc{F}$ is stable under the image $\Gamma_Y$ of $\pi_1(Y)\ra \pi_1(X)\ra \Gamma$.  Let $\Gamma_{\mc{F}}\subset \Gamma_Y$ be the subgroup of elements $\gamma$ such that every fiber in $\mc{F}$ is invariant under $\gamma$. For any $\mu\in \Gamma_Y-\Gamma_{\mc{F}}$,  define $E_{\mu}$ to be the image in $M$ of the union of all fibers which are invariant under $\mu$.  The algebraic subvariety $E_\mu$ is properly contained in $M$. Hence, the $\Gamma_Y$-stabilizer of a very general fiber in $\mc{F}$, i.e. a fiber outside a countable union $\mc{C}_1$ of proper subvarieties of $\mc{F}$, is $\Gamma_{\mc{F}}$.

We first make some observations:

\begin{itemize}
\item There are at most countably many Mumford-Tate domains: we know from \cite[Prop. 6.8]{BBKT} that any Mumford-Tate domain is a component of the Noether-Lefschetz locus 
$\{h'\in \mc{M}: \mcMT_{h'}\subset \mcMT_h\}$
of a mixed Hodge structure $h$. By \cite[Cor. 6.9]{BBKT}, Noether-Lefschetz loci are definable, so they have only finitely many components.
By \cite[Lemma 2.(a)]{A}, any Mumford-Tate group of a mixed Hodge structure $H$ is the largest subgroup of $\tbGL(H_{\QQ})$ that fixes (and thus fixes finitely many) Hodge tensors of $H$.
Recall that Hodge tensors of a mixed Hodge structure $H$ are defined to be the type $(0,0)$ elements of the Hodge structure $T^{m,n}(H_{\QQ}):= H^{\otimes m}\otimes \Hom(H,\ZZ)^{\otimes n}\otimes_{\ZZ}\QQ$, and they are exactly the elements of $F^0T^{m,n}(H_{\CC})\cap T^{m,n}(H_{\QQ})$ of weight $0$, so there are at most countably many Hodge tensors.

\item We have countably many families of weak Mumford-Tate domains described as follows:
the monodromy group $\Gamma$ is countable, so it has at most countably many finitely generated subgroups, and thus it has at most countably many subgroups, denoted by $\Gamma_k$, which are monodromy groups (which are finitely generated because fundamental groups are finitely generated) of smooth locus of closed subvarieties of $X$.
Denote by $\tbP_j$ ($j\in \mathbb{N}$) all the pairwise distinct $\QQ$-groups which arises as the identity component of the $\QQ$-Zariski closures of some $\Gamma_k$.
We index in a way such that $\tbP_0:=\tbP$.
Let $E_j$ be the intersection of $D$ with the union of all Mumford-Tate group orbits of mixed $\mb{Z}$-Hodge structures whose Mumford-Tate group contains $\tbP_j$ as a normal subgroup.
Since there are at most countably many Mumford-Tate domains, we can write $E_j$ as a countable union of intersections of $D$ and some Mumford-Tate domain. 
Since $D$ and Mumford-Tate domains are definable, each such intersection has at most finitely many irreducible components. 
Denote the countably many irreducible components arising in this way by $C_i$.
Let $\mathscr{D}_i$ be the family of $\tbP_j(\RR)^+\tbU_j(\CC)$-orbits of elements in $C_i$, where $\tbU_j$ is the unipotent radical of $\tbP_j$.
Let $\mathscr{D}:=\bigcup_{i=1}^\infty\mathscr{D}_i$.

\item Suppose $X$ is the smallest weakly special subvariety containing $p_X(U)$. 
Then $X$ is the smallest weakly special subvariety containing $p_X(U)^{\Zar}$.
The image of the lifting of $p_X(U)^{\Zar}$ is thus not contained in any weak Mumford-Tate domain in $\mathscr{D}$,
otherwise $\tbP_j=\tbP_0$ for some $j>0$ by Lemma \ref{factor through domain}, which is a contradiction.
Since $U\subset W$, it follows that the image $p_D(U)$ is not contained in any weak Mumford-Tate domain in $\mathscr{D}$.

\item By Theorem \ref{Andre-Deligne: monodromy and proper weakly special}, if the Zariski closure of the projection to $X$ of a component of $V_m\cap W$ has algebraic monodromy group $\tbP'$ strictly smaller than $\tbP$, then it is contained in a proper weakly special subvariety. By Lemma \ref{weak period map}, the image of the lifting of it is contained in the weak Mumford-Tate domain  $D(\tbP')$, 
which is strictly contained in $D$ by Lemma \ref{monodromy and smallest weakly special}. 
 All weak Mumford-Tate domains in $\mathscr{D}$ are strictly contained in $D$.

\item 
We will make use of the fact that $p_D(U)$ is not contained in any weak Mumford-Tate domain in $\mathscr{D}$, in particular $\mathscr{D}_i$.
Firstly, for very general $(V_m,D_{\alpha})\in \mc{F}\times \mathscr{D}_i$,
there is a component of $U_m$ of $V_m\cap W$ with dimension $\geq \dim U$
such that image $p_D(U_m)$ is not contained in $D_{\alpha}$.
Then for very general $V_m$ in $\mc{F}$, 
there is a component $U_m$ of $V_m\cap W$ with dimension $\geq \dim U$
such that image $p_D(U_m)$ is not contained in any domain belonging to $\mathscr{D}_i$.
It follows that for very general $V_m$ in $\mc{F}$, 
there is a component $U_m$ of $V_m\cap W$ with dimension $\geq \dim U$
such that image $p_D(U_m)$ is not contained in any domain belonging to $\mathscr{D}:=\bigcup_{i=1}^\infty\mathscr{D}_i$,
which implies that the image of the lifting of $p_X(U_m)^{\Zar}$ is not contained in any domain in $\mathscr{D}$ (using the fact that $U_m\subset W$).
\end{itemize}

Suppose $X$ is the smallest weakly special subvariety containing $p_X(U)$.  From above, we know that for any $V_m$ outside a countable union $\mc{C}_2$  of proper closed subvarieties of $\mc{F}$, 
there is a component $U_m$ of $V_m\cap W$ with dimension $\geq \dim U$ such that $p_X(U_m)^{\Zar}$ has algebraic monodromy group $\tbP$.
Combining with earlier discussion,
for any  $V_m$ outside $\mc{C}_1\cup \mc{C}_2$, the $\Gamma_Y$-stabilizer of its fiber is $\Gamma_{\mc{F}}$, and there is a component $U_m$ of $V_m\cap W$ with dimension $\geq \dim U$
such that $p_X(U_m)^{\Zar}$ has algebraic monodromy group $\tbP$.
By Lemma \ref{monodromy and smallest weakly special}, $X$ is the smallest weakly special subvariety containing the projection of this component.
By definition of the Hilbert scheme, $\dim V=\dim V_m$.
We have 
$$\dim U_m^{\Zar}-\dim U_m\leq \dim V_m-\dim U_m\leq \dim V-\dim U<\dim \wc{D}.$$
It suffices to prove Ax-Schanuel for fibers outside $\mc{C}_1$; because once this is proved, then the projection of the aforementioned component is contained in a proper weakly special subvariety of $X$, which is a contradiction, and thus $p_X(U)$ is contained in a proper weakly special subvariety of $X$.
Therefore, we can assume $V$ is very general outside $\mc{C}_1$.

\begin{thm}\label{normality of stabilizer}
The subgroup $\tbN$ is normal in $\tbP$.
\end{thm}

\begin{proof}
Since $q$ is definable, each fiber of $q$  has only finitely many components. Then $\Gamma_Y$ is of finite index in $\Gamma$. Since $\Gamma$ is Zariski-dense in the connected group $\tb{P}$, it follows that  $\Gamma_Y$ is Zariski-dense in $\tbP$. Every element $\gamma\in \Gamma_Y$ sends a very general fiber of $\mc{F}$  to a very general fiber, so $Stab(V)=\Gamma_{\mc{F}}=Stab(\gamma V)=\gamma Stab(V)\gamma^{-1}$.  Since $\Gamma_Y$ is Zariski-dense in $\tbP$, $\tbN$ is then normal in $\tbP$.
\end{proof}

\section{Definable quotient $\ol{I}$}\label{definable subset I}

As explained in the idea of proof (Section \ref{Idea of proof}), we have to look at the projection of $r(D)$ onto the reductive part modulo the stabilizer part, which will be defined and denoted by $D_{\tbL}$.

\subsection{Definable fundamental set $\Phi$ for $\tb{P}(\mb{Z})^+\bs D$}\label{constructing definable fundamental set}

Let  $\tbN_u$ be the unipotent radical of $\tbN$.  Recall in Section \ref{Real split locus of domains} we fixed Levi subgroup $\tbN_r$ in $\tbN$  and  Levi subgroup $\tbG$ in $\tbP$  such that $\tbN_r\subset \tbG$. 
By Theorem \ref{normality of stabilizer}, $\tbN$ is normal in $\tbP$,  so $\tbN_u=\tbN\cap \tbU$ and $\tbN_r=\tbN\cap \tbG$  by \cite[Prop. 2.13]{G2}.  Thus the $\mb{Q}$-group $\tbN_r$ is normal in $\tbG$. Similarly, $\tbG$ is normal in $\mcMT_r$. By \cite[Cor. 2]{A}, $\tbG$ is semisimple, so there exists a connected normal subgroup $\tbL$ of $\tbG$ such that the map $\tbN_r\times\tbL\ra \tbG$ defined by $(g_1, g_2)\mapsto g_1g_2$ is an isogeny \cite[Theorem 21.51]{Mil2}. This induces an isogeny $\beta:\tbL \ra \tbG/\tbN_r$.  
Let $D_{\tbN_r}$ and $D_{\tbL}$ be respectively the $\tbN_r(\mb{R})^+$-orbit  and the $\tbL(\mb{R})^+$-orbit of the pure Hodge structure ${h_1}_{\text{Gr}}$ in $D_{\Gr}$, where ${h_1}_{\text{Gr}}$  was defined in Section \ref{Real split locus of domains}.
By \cite[II. B]{GGK},  we have an isomorphism $D_{\tb{N}_r}\times D_{\tbL}\simeq D_{\text{Gr}}$.
Combining this with the isomorphism in  Theorem \ref{splitting of retraction of weak mixed domain}, we have an isomorphism 
$$j: r(D)\simeq D_{\tb{N}_r}\times D_{\tbL}\times D_{\tb{U},\mb{R}}.$$
The unipotent radical acts trivially on the associated graded, so for any $u\in \tbU(\mb{R})$, $g_1\in\tbN_r(\mb{R})^+$, and $g_2\in \tbL(\mb{R})^+$, we have 
$$j(ug_1g_2h_1)=(g_1{h_1}_{\text{Gr}}, g_2{h_1}_{\text{Gr}}, ug_1g_2{h_1}_s).$$

Let $\tbN_r(\ZZ)^+:=\tbN_r(\ZZ)\cap \tbN_r(\RR)^+$ and $\tbL(\ZZ)^+:=\tbL(\ZZ)\cap \tbL(\RR)^+$.
\begin{lemma}\label{image of arithmetic subgroup}
We have 
$\tbP(\mb{Z})^+=\bigcup_{i=1}^k  \tbU(\mb{Z})\tbN_r(\mb{Z})^+\tbL(\mb{Z})^+\rho_i$
for some $\rho_1,\dots, \rho_k\in \tbP(\mb{Z})^+$, where one of the $\rho_i$ is the identity.
\end{lemma}

\begin{proof}
By \cite[p. 173, Cor. 2]{PR}, $\tbU(\mb{Z})\tbG(\mb{Z})$ is of finite index in $\tbP(\mb{Z})$, so $\tbU(\ZZ)\tbG(\ZZ)^+$ is of finite index in $\tbP(\ZZ)^+$. Consider the isogeny $\tbN_r\times \tbL\ra \tbG$. By \cite[Theorem 4.1]{PR}, $\tbN_r(\mb{Z})\tbL(\mb{Z})$ is of finite index in $\tbG(\mb{Z})$, so $\tbN_r(\ZZ)^+\tbL(\ZZ)^+$ is of finite index in $\tbG(\ZZ)^+$.
\end{proof}

Recall the fundamental set  $\Phi_{\tbU}$ for the $\tbU(\mb{Z})$-action on $D_{\tbU,\mb{R}}$ and recall the arithmetic subgroup $\Lambda$ of $\tbG(\mb{Q})$ in Section \ref{temporary definable fundamental set}. 
Let $\Phi_{\tbL}$ and $\Phi_{\tb{N}_r}$  be definable open fundamental sets for the actions of $\tbL(\mb{Z})^+$ and $\tb{N}_{r}(\mb{Z})^+$ on $D_{\tbL}$ and $D_{\tb{N}_r}$  respectively  \cite[Theorem 1.1]{BKT}. We can assume that $\Phi_{\tbL}$ and $\Phi_{\tbN_r}$ contain ${h_1}_{\text{Gr}}$. Since $\tbN_r(\mb{Z})^+\tbL(\mb{Z})^+$ is of finite index in $\Lambda$ \cite[Theorem 4.1]{PR}, the image of $\Phi_{\tbL}\times \Phi_{\tbN_r}$ in $D_{\Gr}$ is a definable open fundamental set for the action of $\Lambda$ on $D_{\Gr}$. By Lemma \ref{product of fundamental sets is a fundamental set} with $\Phi_{\tbG}$ replaced by $\Phi_{\tb{N}_r} \times\Phi_{\tbL}$, the set $j^{-1} (\Phi_{\tb{N}_r} \times\Phi_{\tbL}\times  \Phi_{\tb{U}})$  is a definable open fundamental set for the action of $\tb{P}(\mb{Z})^+$ on $r(D)$, so 
$$\Phi_{\mb{R}}:=  \bigcup_{i=1}^k \rho_i\cdot   j^{-1} (\Phi_{\tb{N}_r} \times\Phi_{\tbL}\times  \Phi_{\tb{U}})$$
is also a definable open fundamental set for the action of $\tb{P}(\mb{Z})^+$ on $r(D)$.
Then $\Phi:= r^{-1}(\Phi_{\mb{R}})$ is a definable open fundamental set for the action of $\tbP(\mb{Z})^+$ on $D$.

\subsection{Definable subset $I$ of $\textbf{P}(\CC)$}\label{definable subset I of P(R)}

Let
$$I:=\{\gamma\in \tbP(\RR): \dim \gamma^{-1}V \cap W\cap (X\times \Phi)\geq  \dim U\},$$
where $\gamma^{-1}$ acts on $V$ by acting on the $\widecheck{D}$-coordinates. By \cite[Proposition 2.3]{BBKT}, $\tbP(\mb{Z})^+\bs D$ has a definable structure such that the canonical map $\Phi\ra \tbP(\mb{Z})^+\bs D$ is definable. 
By Lemma \ref{definability varphi}, the set 
$$W\cap (X\times \Phi)=\{(x,F^\bullet)\in X\times \Phi: \varphi(x)=\pi|_{\Phi} (F^\bullet)\}$$
is then definable. Then since $V$ is algebraic and since $\dim U$ is a fixed number, $I$ is definable.

\subsection{Definable quotient $\ol{I}$}
Recall that $\tbN$ is the identity component of the $\mb{Q}$-Zariski closure of $Stab(V):=\{\sigma\in \tb{P}(\mb{Z})^+: \sigma V=V\}$ in $\tb{P}$, so $\tbN_{\mb{C}}$ is the identity component of the $\mb{C}$-Zariski closure of $Stab(V)$ by Lemma \ref{Galois descent, Zariski closure}. Moreover, $V$ is algebraic and invariant under $Stab(V)$, so $V$ is invariant under  $\tbN_{\mb{C}}$. Let $\ol{I}$ be the definable image of $I$ under the map $\tbP(\RR)\ra (\tbP/\tbN)(\RR)$, i.e. 
$$\ol{I}:=\{[\gamma]: \gamma\in \tbP(\RR)\text{ and } \dim \gamma^{-1}V \cap W\cap (X\times \Phi)\geq  \dim U\}.$$
Let  $\tbW$  be  the unipotent radical of $\tbP/\tbN$.  By \cite[Corollary 14.11]{Bor}, $\tb{W}=\tb{U}\tb{N}/\tb{N}$. The group $\tbH:=\tb{GN}/\tbN$ is a Levi subgroup of $\tbP/\tbN$. We have $\tbH\cong \tbG/(\tbG\cap \tbN)$.
The image  $I_{\tbW}$ of $\ol{I}$ under the definable projection $\pi: (\tbP/\tbN)(\mb{R})\simeq \tbW(\mb{R})\rtimes \tbH(\mb{R})\ra \tbW(\mb{R})$ is definable. We have
\begin{align*}
I_\tb{W}=\{[\gamma]: &\gamma\in \tb{U}(\mb{R})\text{ and } \dim (\gamma\eta)^{-1}V \cap W\cap (X\times \Phi)\geq  \dim U
\\& \text{ for some } \eta \in \tb{L}(\RR) \}.
\end{align*}

Let $\tbN(\ZZ)^+:=\tbN(\ZZ)\cap \tbP(\RR)^+$.

\begin{lemma}\label{intersecting fundamental domain, definable quotient}
If  $\gamma$  is in $\tbP(\mb{Z})^+$ such that $U\cap (X\times \gamma\bigcup_{\sigma\in \tbN(\mb{Z})^+} \sigma \Phi)\neq \varnothing$, then $\dim \gamma^{-1}V \cap W\cap (X\times \Phi)\geq  \dim U$.
\end{lemma}

\begin{proof}
We have $ \dim U\cap (X\times \gamma\sigma\Phi)=\dim U$ for some $\sigma\in \tbN(\mb{Z})^+$. 
Then 
\begin{align*}
\dim \gamma^{-1}V \cap W\cap (X\times \Phi)
&=\dim \gamma^{-1}(\sigma^{-1} V)\cap W\cap (X\times \Phi)
\\&=\dim (\sigma^{-1}\gamma^{-1}) V\cap W\cap (X\times  \Phi)
\\&=\dim V\cap W\cap (X\times \gamma\sigma\Phi)
\\&\geq \dim U\cap (X\times \gamma\sigma\Phi)
\\&= \dim U.\qedhere
\end{align*}
\end{proof}

\begin{lemma}\label{inifinite graph, AS}
If the set $\ol{I}$ contains a semialgebraic curve, then Theorem \ref{Ax-Schanuel for variations of mixed Hodge structures} holds. Similarly, if the set  $I_{\tbW}$ contains a semialgebraic curve, then Theorem \ref{Ax-Schanuel for variations of mixed Hodge structures} holds.
\end{lemma}

\begin{proof}
The group $\tbN(\RR)$ is of finite index, say $q$, in the $\RR$-Zariski closure of $Stab(V)$. Choose $p>q$.
Let $C_{\RR}$ be a semialgebraic curve in $\ol{I}$. It contains at least $p$ points. Let $C$ be a complex algebraic curve containing $C_{\mb{R}}$.
By definition of $\ol{I}$, for each $[c]\in C_{\mb{R}}$, there exists  an irreducible analytic component of $c^{-1}V\cap W\cap (X\times \Phi)$ of dimension at least $\dim U$.
It follows that there exists $[c_0]\in C_{\RR}$ and an irreducible analytic component $U_0$ of $c_0^{-1}V\cap W\cap (X\times \Phi)$ such that the dimension stays at least $\dim U$ as $c_0$ varies outside a countable subset of $C$.
Let $V^\prime$ be the smallest algebraic variety containing $C^{-1}V$.
Let $U'$ be the irreducible analytic component of $V'\cap W\cap (X\times \Phi)$ containing $U_0$.

Since the curve $C_{\RR}$ contains at least $p$ points, and since $p>q$, we have $C^{-1}V\neq V$. Hence, if $U_0\subset c^{-1}V$ as $c$ varies, then 
$$\dim U\leq \dim U_0\leq  \dim U_0^{\Zar}\leq \dim \bigcap_{c\in C}c^{-1}V< \dim V.$$
By induction hypothesis, Ax-Schanuel holds if we replace $V$ and $U$ by $U_0^{\Zar}$ and $U_0$,
so $p_X(U_0)$ is contained in a proper weakly special subvariety. 
Otherwise, $U_0\not\subset c^{-1}V$ as $c$ varies, then $c^{-1}V\cap W\neq c_0^{-1}V\cap W$ as $c$ varies.
Since $C$ is semi-algebraic and since $C^{-1}V\neq V$, we have $\dim V'=\dim V+1$. 
Moreover, $\dim U'\geq  \dim U+1$.
By induction hypothesis, Ax-Schanuel holds if we replace $V$ and $U$ by $U'^{\Zar}$ and $U'$.
Thus $p_X(U_0)$, which a subset of $p_X(U')$, is contained in a proper weakly special subvariety. 
We conclude that for $c$ outside a countable subset of $C$, there is an irreducible component $U_c$ of $c^{-1}V\cap W\cap (X\times \Phi)$ with dimension $\geq \dim U$ such that $p_X(U_c)$ is contained in a proper weakly special subvariety.

By translating $U$ by an element in $\tbP(\ZZ)^+$, it suffices to prove Theorem \ref{Ax-Schanuel for variations of mixed Hodge structures} for the case where $U\cap (X\times \Phi)\neq \varnothing$.
Suppose $p_X(U)$ is not contained in any proper weakly special subvariety. 
By Theorem \ref{Andre-Deligne: monodromy and proper weakly special}, $p_X(U)^{\Zar}$ has algebraic monodromy group $\tbP$.
Then by the same argument as in the discussion before Theorem \ref{normality of stabilizer},
for $c$ outside a countable subset of $C$,
there is an irreducible component $U_c$ of $c^{-1}V\cap W\cap (X\times \Phi)$ with dimension $\geq \dim U$ such that $p_X(U_c)^{\Zar}$ has algebraic monodromy group $\tbP$,
which implies that $p_X(U_c)$ is not contained in any proper weakly special subvariety by Theorem \ref{Andre-Deligne: monodromy and proper weakly special}.
This leads to a contradiction, so $p_X(U)$ is contained in a proper weakly special subvariety, as desired.

Let $C_W$ be the semialgebraic curve in $I_W$. It contains at least $p$ points.
The preimage of $C_W$ under the definable projection $\pi: \ol{I}\ra I_W$ is a semialgebraic set containing at least $p$ points.
By intersecting this semialgebraic set with other semialgebraic sets interpolating these points, $\ol{I}$ contains a semialgebraic curve containing at least $p$ points.
The second statement then follows from the first.
\end{proof}

\section{Heights and trichotomy}\label{Heights}
We will define the height of a subset of $r(D)$. After that, we can then apply Gao's mixed point counting method \cite[Theorem 5.2]{G} to get a trichotomy.

Fix an embedding $\ol{\phi}: \tbP/\tbN\hookrightarrow \tbGL_{m}$ for some $m$.   By conjugation, we can assume $\tbW$ is mapped by $\ol{\phi}$ into the $\mb{Q}$-group $\mb{U}_m$ of upper unitriangular $m\times m$ matrices. Let $\phi: \tb{P}\hookrightarrow \tb{GL}(\mc{H}_{\mb{Q},\eta})\cong \tb{GL}_{\ell}$ be the inclusion followed by an isomorphism, where $\ell:=\dim \mc{H}_{\mb{Q},\eta}$.

\begin{definition}
For any rational square matrix $A$, define the \textbf{height} $\hgt A$ of $A$ to be the maximum of the naive heights of the entries. For any $[\gamma]\in (\tbP/\tb{N})(\mb{Q})$, define the \textbf{height} $\hgt [\gamma]$ of $[\gamma]$ to be $\hgt \ol{\phi}([\gamma])$.  For any $\gamma\in \tbP(\mb{Q})$, define the \textbf{height} $\hgt \gamma$ of $\gamma$ to be $\hgt \phi(\gamma)$. 
\end{definition}

\begin{lemma}\label{height of inverse}
Let $k$ be a positive integer. For any $k\times k$ integer invertible matrix $A$, we have $\hgt A^{-1}\leq (k-1)!(\hgt A)^{k-1}.$
\end{lemma}

\begin{proof}
Denote the $ij$-minor by  $M_{ij}$. Since $\det A =\pm 1$, $\hgt A^{-1}=\hgt \tb{adj} A= \max_{i.j} | \det M_{ij}|\leq (k-1)!(\hgt A)^{k-1}$.
\end{proof}

Recall the isogeny $\beta:\tbL \ra \tbG/\tbN_r$ in Section \ref{constructing definable fundamental set}. Recall that $\tbN_r=\tbN\cap \tbG$. The map $\alpha:\tbG/\tbN_r\ra \tbG\tbN/\tbN=: \tbH$  defined by $g\tbN_{\tbG}\mapsto g\tbN$ is an isomorphism. The map $\tau:=\alpha\circ \beta$ is an  isogeny.

\begin{lemma}\label{comparing heights, algebraic}
There exist constants $k_1,k_2>0$ such that $ \hgt \tau(\gamma)\leq k_1(\hgt\gamma)^{k_2}$ for any $\gamma\in \tbL(\mb{Q})$.
\end{lemma}

\begin{proof}
The isogeny $\tau$,  and the embeddings $\ol{\phi}, \phi$ in the definitions of heights, are algebraic. 
\end{proof}

For any nonempty $E\subset r(D)$, define 
$$\hgt_{\tbW} E:= \max\{\hgt [\gamma]: \gamma\in \tb{U}(\mb{Z}),  E\cap  \gamma\bigcup_{\substack{ \eta\in \tb{L}(\mb{Z})^+,\sigma \in \tb{N}_r(\mb{Z})^+}}  \sigma \eta\Phi_{\mb{R}} \neq \varnothing \},$$
which is allowed to be infinite.

 Consider the projection   $p_{\tb{L}}:r(D)\ra D_{\tbL}$.  Let $Z:=r(p_D(U))$. Since $\Phi_{\RR}$ is a fundamental set for the action of $\tbP(\ZZ)^+$ on $r(D)$, by translating $U$ by an element in $\tbP(\ZZ)^+$, we can assume that $Z\cap \Phi_{\RR}\neq \varnothing$. Fix a regular point $h$ in  $Z\cap \Phi_{\mb{R}}$. Denote the radius $T$ ball in $D_{\tbL}$ centered at a point $p_{\tbL}(h)$ by $B_{p_{\tbL}(h)}(T)$.  Let $Z(T)$ be the irreducible analytic component of $Z\cap p_{\tbL}^{-1}(B_{ p_{\tbL}(h)}(T))$ which contains $h$.

\begin{lemma}\label{height bound and radius, with denominators}
For any irreducible component $E$ of $p_{\tbL}(r(\pi^{-1}(\varphi(X))))$, there exists a constant $c_1>0$ such that for any $T\gg 0$, if 
$ \eta\Phi_{\tbL}\cap B_{p_{\tbL}(h)}(T)\cap E\neq \varnothing$ 
for some $\eta\in \tbL(\mb{Z})^+$, then $\hgt \eta\leq e^{c_1T}$. 
\end{lemma}

\begin{proof}
This follows from \cite[Theorem 4.2]{BT} (note that $\Phi_{\tbL} \cap E$ overlaps with only finitely many translates of the fundamental set in \cite{BT}) and Lemma \ref{height of inverse}.
\end{proof}

\paragraph{\textbf{Trichotomy}}
Fix a number $\lambda>2k_2c_12^mm$. We are in one of the following three cases:
\begin{enumerate}
\item  We have $\dim p_{\tbL}(Z)>0$, and for some sequence $\{T_i\in  \mb{R}\}_{i\in \mb{N}}$ such that $T_i\ra\infty$, we have $\hgt_{\tbW} Z(T_i)\leq e^{\lambda T_i}$ for all $i$.
\item We have $\hgt_{\tbW} Z(T)> e^{\lambda T}$ for all $T\gg 0$. (This includes the case where $\dim p_{\tbL}(Z)=0$ and $\hgt_{\tbW} Z$ is infinite, because if $\dim p_{\tbL}(Z)=0$, then $Z\subset p_{\tbL}^{-1}(B_{ p_{\tbL}(h)}(T))$ for all $T$. Thus, for all $T$, $\hgt_{\tbW} Z(T)=\hgt_{\tbW} Z$, which is infinite.)
\item We have $\dim p_{\tbL}(Z)=0$ and $\hgt_{\tbW} Z$ is finite.
\end{enumerate}

\section{Proof of Case (1)}\label{Proof of Case (1)}
The main idea, borrowed from Gao \cite{G}, of the proof of the following theorem is to use the volume estimates established by Bakker-Tsimerman \cite{BT} to produce enough reductive integer points, and attach to each of these points a unipotent integer point of comparable (or smaller) heights using the assumption of case (1).

\begin{thm}\label{Case 1: point count}
Suppose we are in case (1), described in the previous section. There exist  constants $c_3,c_4>0$ such that for any $i\gg 0$, there exists at  least $c_3T_i^{c_4}$  rational points $[\rho]\in (\tbP/\tbN)(\mb{Q})$ of heights at most $T_i$ such that $U\cap (X\times \rho^\prime  \Phi)\neq \varnothing$ for some $\rho^\prime\in [\rho]\cap \tbP(\mb{Z})^+$.
\end{thm}

\begin{proof}
Suppose we have an element $\eta$ in $\tbL(\mb{Z})^+$ such that $p_{\tbL}(Z(T_i))\cap \eta\Phi_{\tbL}\neq \varnothing$. Let $z_{\tbL}\in p_{\tbL}(Z(T_i))\cap \eta \Phi_{\tbL}$. Write $p_{\tbL}(z)=z_{\tbL}$ for some $z\in Z(T_i)$. Let $(z_{\tbN_r}, z_{\tbL},z_{\tb{U}})$ be the image of $z$ under the isomorphism 
$$j:r(D)\simeq D_{\tb{N}_r}\times D_{\tbL}\times D_{\tb{U},\mb{R}}.$$
There exists $\sigma\in \tbN_{r}(\mb{Z})^+$ such that $z_{\tbN_r}\in \sigma\Phi_{\tbN_r}$.  Recall that $\tbP(\mb{R})^+$ acts on $D_{\tbU,\mb{R}}$ \emph{a priori}.  There exists $\gamma\in \tb{U}(\mb{Z})$ such that  $z_{\tbU}\in \gamma\sigma\eta\Phi_{\tbU}$.
Recall that we fixed $h$ in  $Z\cap \Phi_{\mb{R}}$ in Section \ref{Heights}.
For some $g_1\in \tbN_r(\mb{R})^+$, $g_2\in \tbL(\mb{R})^+$ and $u\in \tbU(\mb{R})$, we have $g_1h_{\text{Gr}}\in \Phi_{\tbN_r}$, $g_2h_{\text{Gr}}\in \Phi_{\tbL}$, $uh_s\in \Phi_{\tbU}$, and
 $$ (z_{\tbN_r}, z_{\tbL},z_{\tb{U}})=(\sigma g_1h_{\text{Gr}}, \eta g_2h_{\text{Gr}}, \gamma\sigma\eta uh_s).$$
Let $g:=g_1g_2$.
By Lemma \ref{retraction of D is transitive}, there exists $u^\prime\in\tbU(\mb{R})$ such that $g^{-1}{h}_s=u^\prime {h}_s$. 
Then since $\tbU(\mb{R})$ acts trivially on the associated graded, we have 
$$ugu^\prime h =j^{-1}((g_1{h}_{\text{Gr}}, g_2{h}_{\text{Gr}}, ugu^\prime {h}_s))\in \Phi_{\mb{R}},$$
and thus (recall that $\tbN_r\times \tbL\ra \tbG$ in Section \ref{constructing definable fundamental set} is an isogeny, and hence a homomorphism)
$$z=j^{-1}((z_{\tbN_r}, z_{\tbL},z_{\tb{U}}))=\gamma\sigma \eta ugu^\prime h\in \gamma \sigma\eta \Phi_{\mb{R}}.$$
 Then $z\in Z(T_i)\cap \gamma\sigma\eta \Phi_{\mb{R}}$.
By Lemma \ref{comparing heights, algebraic} and Lemma \ref{height bound and radius, with denominators}, 
$$\hgt[\eta]=\hgt\tau(\eta)\leq k_1(\hgt \eta)^{k_2}\leq k_1 e^{k_2 c_1T_i}.$$
By $\tbN_r\subset \tbN$ and assumption, $\hgt[\gamma\sigma]=\hgt [\gamma]\leq e^{\lambda T_i}$. It follows that 
$$\hgt[\gamma\sigma\eta]\leq  2^{m-1}(\hgt [\gamma\sigma]\hgt[\eta])^m =O(e^{m(\lambda+k_2c_1) T_i}).$$

Since $\dim p_{\tbL}(Z)>0$, by \cite[Theorem 1.2]{BT} and  \cite[Proposition 3.2]{BT}, there exists a constant $c>0$ such that for any $T>0$, there exist at least $e^{cT}$ integer points $\eta$ in $\tbL(\mb{Z})^+$ of heights at most $e^{c_1T}$ such that $p_{\tbL}(Z(T))\cap \eta\Phi_{\tbL}\neq \varnothing$. 

Combining with what we have proved above, taking into consideration that $\tau$ is an isogeny, we know there exist constants $c_3,c_4>0$ such that,  for any $i\geq 0$,  there exist at least $c_3T_i^{c_4}$ points $[\rho]$ in $(\tbP/\tbN)(\mb{Q})$ of heights at most $T_i$ such that $Z(T_i)\cap \rho^\prime \Phi_{\mb{R}}\neq \varnothing$ for some $\rho^\prime\in [\rho]\cap \tbP(\mb{Z})^+$. Since $Z=r(p_D(U))$ and $r$ is $\tb{P}(\mb{Z})^+$-equivariant, the theorem follows.
\end{proof}

\begin{thm}
Theorem \ref{Ax-Schanuel for variations of mixed Hodge structures} holds in case (1).
\end{thm}

\begin{proof}
By Theorem \ref{Case 1: point count} and Lemma \ref{intersecting fundamental domain, definable quotient}, for any $i\gg 0$, $\ol{I}$ contains at least $c_3T_i^{c_4}$  rational points  $[\rho]$ of heights at most $T_i$, where $\rho\in  \tbP(\mb{Z})^+$. Since $\ol{I}$ is definable, by  the counting theorem of Pila-Wilkie \cite{PW}, and the assumption that $T_i\ra \infty$, for any positive integer $p$, $\ol{I}$ contains a semialgebraic curve containing at least $p$ points of the form $[\gamma]$, where $\gamma\in \tbP(\ZZ)^+$. By Lemma \ref{inifinite graph, AS}, the theorem follows.
\end{proof}

 \section{Proof of case (2)}\label{Proof of case (2)}
The idea of the proof of this case is to first define a family of connected graphs encoding how $Z$ is intersecting the translates of the fundamental sets. Next we understand how walking along a path in the graph is the same as multiplying conjugates of unitriangular matrices. Then once we get an upper estimate of the height of the product of these conjugates, we can get enough points in the graphs. We then have enough points in the unipotent projection $I_{\tbW}$ of $\ol{I}$.

 For any $T>0$, let $Q_T$ be a graph with vertex set and edge set as follows:
 $$V(Q_T):=\{[\gamma]: \gamma\in \tbU(\mb{Z}), Z(T)\cap  \gamma\bigcup_{\substack{\eta\in \tb{L}(\mb{Z})^+,\sigma \in \tb{N}_r(\mb{Z})^+}}  \sigma\eta \Phi_{\mb{R}}\neq \varnothing \},$$
\begin{align*}
&E(Q_T):=\{([\gamma_{1}],[ \gamma_{2}]):  \gamma_{1},\gamma_{2}\in \tbU(\mb{Z}),
\\&Z(T)\cap (\gamma_{1}\bigcup_{\substack{ \eta\in \tb{L}(\mb{Z})^+,\sigma \in \tb{N}_r(\mb{Z})^+}} \sigma\eta \Phi_{\mb{R}})
\cap ( \gamma_{2}\bigcup_{\substack{\eta\in \tb{L}(\mb{Z})^+,\sigma \in \tb{N}_r(\mb{Z})^+}}  \sigma \eta\Phi_{\mb{R}})\neq \varnothing
\}. 
\end{align*}

\begin{lemma}\label{graph connected}
The graph $Q_T$ is connected.
\end{lemma}

\begin{proof}
Pick any vertices $[\gamma_{1}],[ \gamma_{2}]$ of $Q_T$. For each $j=1,2$, choose  
$$x_j\in  Z(T)\cap \gamma_{j}\bigcup_{\substack{ \eta\in \tb{L}(\mb{Z})^+,\sigma \in \tb{N}_r(\mb{Z})^+}}\sigma\eta \Phi_{\mb{R}}.$$
By Lemma \ref{image of arithmetic subgroup} and the definition of $\Phi_{\mb{R}}$,
$$r(D)=\bigcup_{\rho\in \tbP(\mb{Z})^+} \rho\cdot   j^{-1} (\Phi_{\tb{N}_r} \times\Phi_{\tbL}\times  \Phi_{\tb{U}})=\bigcup_{\substack{\gamma\in \tbU(\mb{Z}),  \eta\in\tb{L}(\mb{Z})^+, \sigma\in \tbN_r(\mb{Z})^+}}  \gamma\sigma \eta\Phi_{\mb{R}}.$$
 Since $Z(T)$ is path-connected, there exists a path in 
$$Z(T)=\bigcup_{\{\gamma\in \tbU(\mb{Z}): [\gamma]\in V(Q_T)\}} \bigcup_{\substack{\eta\in\tb{L}(\mb{Z})^+, \sigma\in \tbN_r(\mb{Z})^+}} Z(T)\cap \gamma  \sigma\eta \Phi_{\mb{R}}$$
joining $x_1$ and $x_2$. Since $\Phi_{\RR}$ is an open fundamental set, this induces a path in the graph $Q_T$ joining $[\gamma_{1}]$ and $[ \gamma_{2}]$. It follows that $Q_T$ is connected.
\end{proof}

Let $S:=\{\delta \in\tb{P}(\mb{Z})^+: \delta \Phi_{\mb{R}}\cap \Phi_{\mb{R}}\neq\varnothing\}$, which is a finite set by the definition of a fundamental set. For any $\delta\in \tbP(\mb{Z})^+$,  write $[\delta]=[\delta]^{\tbW}[\delta]^{\tbH}$, where  $[\delta]^{\tbW}\in \tbW(\mb{Q})$ and $[\delta]^{\tbH}\in \tbH(\mb{Q})$. 

\begin{lemma}\label{adjacent vertices and heights}
When  $T\gg 0$, the following holds: Suppose  $[\gamma_{1}]$ and $[\gamma_{2}]$ are adjacent vertices in $Q_T$.   Then
$[\gamma_{2}]=[\gamma_{1}][\eta_{1}][\delta]^{\tbW}[\eta_{1}]^{-1}$
for some  $\delta\in S$ and some $\eta_{1}\in \tbL(\mb{Z})^+$
satisfying $\hgt[\eta_1]\leq e^{2k_2c_1T}$, where $k_2$ is in Lemma  \ref{comparing heights, algebraic} and $c_1$ is  in Lemma \ref{height bound and radius, with denominators}, both independent of $T$.
\end{lemma}

\begin{proof}
There exist   $\sigma_1,\sigma_2\in \tbN_r(\mb{Z})^+$ and $\eta_{1}, \eta_{2}\in \tbL(\mb{Z})^+$  such that
$$Z(T)\cap  \gamma_{1} \sigma_1\eta_{1}\Phi_\mb{R}\cap  \gamma_{2} \sigma_2\eta_{2}\Phi_\mb{R}\neq \varnothing.$$
Then
$$\Phi_\mb{R}\cap   ((  \gamma_{1} \sigma_1\eta_{1})^{-1}  \gamma_{2} \sigma_2\eta_{2})\Phi_\mb{R}\neq \varnothing,$$ 
 so 
 $\gamma_{2}\sigma_2\eta_{2}=  \gamma_{1}\sigma_1\eta_{1}\delta$ for some $\delta\in S$. 
 We then have 
$$[\gamma_{2}][\eta_{2}]=[\gamma_{1}][\eta_{1}][\delta]
=[\gamma_{1}][\eta_{1}][\delta]^{\tbW}[\eta_{1}]^{-1}[\eta_{1}][\delta]^{\tbH}.$$
Since $\tb{P}/\tbN= \tbW\rtimes\tbH$, we have $[\gamma_{2}]=[\gamma_{1}][\eta_{1}][\delta]^{\tbW}[\eta_{1}]^{-1}$.

Let $z\in Z(T)\cap  \gamma_{1}\sigma_1\eta_{1}\Phi_\mb{R}$. Write $z=\gamma_1\sigma_1\eta_1p{h}$ for some $p\in \tbP(\mb{R})^+$. Write $p=ug_1g_2$ for some $u\in \tbU(\mb{R})$, $g_1\in \tbN_r(\mb{R})^+$ and $g_2\in \tbL(\mb{R})^+$. 
Since the unipotent radical acts trivially on the associated graded, $g_2{h}_{\text{Gr}}\in \Phi_{\tbL}$.
Thus $p_{\tbL}(z)=\eta_1g_2{h}_{\text{Gr}}\in  \eta_1\Phi_{\tbL}$ (recall that $\tbN_r\times \tbL\ra \tbG$ in Section \ref{constructing definable fundamental set} is an isogeny, and hence a homomorphism),
so $\eta_1\Phi_{\tbL}\cap p_{\tbL}(Z(T))\neq \varnothing$. By Lemma \ref{height bound and radius, with denominators}, $\hgt  \eta_1\leq e^{c_1T}$. Hence by Lemma \ref{comparing heights, algebraic}, $\hgt[\eta_1]=\hgt \tau(\eta_1)\leq k_1(\hgt \eta_1)^{k_2}\leq e^{2k_2c_1T}$ when $T\gg 0$.
\end{proof}

\begin{lemma}\label{product of unipotent matrix, clearing denominators}
Let $\mc{Y}$ be a subset of the group $\mb{U}_m(\mb{Q})$ of upper unitriangular (upper triangular with 1's on the diagonal) $m\times m$ rational matrices  such that the heights of matrices in $\mc{Y}$ are bounded by a constant.  Let 
$s_1:=\min\{ d\in \mb{Z}^+: dY\in \Mat_m(\mb{Z}) \text{ for all } Y\in \mc{Y}\}$.  For any positive integer $t$, and any $Y_1,\dots, Y_t\in \mc{Y}$, we have  $s_1^{m-1}Y_1\cdots Y_t\in \Mat_m(\mb{Z})$.
\end{lemma}

\begin{proof}
We prove by induction on $m$. By induction assumption,  $s_1^{m-2}\cdot  (Y_1\cdots Y_t)_{ij}\in \mb{Z}$ for any $(i,j)\neq (1,m)$. Moreover, 
$$ s_1^{m-1}\cdot (Y_1\cdots Y_t)_{1,m}=\sum_{i=1}^t s_1^{m-1}\cdot  (Y_i)_{1,m}+\sum_{k=2}^t\sum_{i=2}^{m-1} s_1^{m-2}\cdot (Y_1\cdots Y_{k-1})_{1,i} \cdot s_1\cdot (Y_k)_{i,m} $$
is an integer.
\end{proof}

\begin{lemma}\label{unipotent matrix height bound}
Let $\mc{S}$ be a finite subset of the group $\mb{U}_m(\mb{Q})$ of upper unitriangular $m\times m$ rational matrices  such that the heights of matrices in $\mc{S}$ are bounded by a constant $s_0$. Let $f$ be a positive integer.  There exists a constant $C>0$, depending only on $\mc{S}, m, f$ and $s_0$, such that the following holds: 

Let $l_0>0$. Let $\mc{L}\subset \tbGL_{m}(\mb{Q})$ such that $fB\in \Mat_m(\mb{Z})$ and $\hgt B\leq l_0$ for all $B\in \mc{L}$. Let $A_1,\dots, A_r\in \mc{S}$ and $B_1,\dots, B_r\in \mc{L}$ satisfying $B_jA_jB_j^{-1}\in \mb{U}_m(\mb{Q})$ for all $j$.   Then 
$$\hgt\left(\prod_{j=1}^r B_jA_jB_j^{-1}\right) \leq C \cdot l_0^{2^mm} \cdot r^{m-1}.$$
\end{lemma}

\begin{proof}
We prove by induction on $m$. The case when $m=1$ is trivial because $B_jA_jB_j^{-1}=(1)$ for all $j$.  Let $\mc{S}^\prime$ (resp. $\mc{S}^{\prime\prime}$) be the set of all $(m-1)\times (m-1)$ upper unitriangular matrices  obtained by deleting the last (resp. first) row and the last (resp. first) column of matrices in $\mc{S}$.  By induction assumption, we have a constant $C^\prime>0$ (resp. $C^{\prime\prime}>0$) such that the lemma holds with $m$ and  $\mc{S}$ being replaced by $m-1$ and  $\mc{S}^\prime$ (resp. $\mc{S}^{\prime\prime}$). Let $\mc{L}^\prime$ (resp. $\mc{L}^{\prime\prime}$) be the set of all $(m-1)\times (m-1)$ matrices obtained by deleting the last (resp. first) row and the last (resp. first) column of matrices in $\mc{L}$. Write $Y_j:=B_jA_jB_j^{-1}\in \mb{U}_m(\mb{Q})$. 
Define $s_1:=\min\{ s\in \mb{Z}^+: sA\in \Mat_m(\mb{Z}) \text{ for all } A\in \mc{S}\}$. 
By Lemma \ref{height of inverse}, for all $j$, 
\begin{align*}
\hgt Y_j\leq  s_1\hgt (fB_j)(s_1A_j)(fB_j)^{-1}&\leq s_1m^2(m-1)!\hgt(s_1A_j)(\hgt(fB_j))^m
=O(l_0^m).
\end{align*}
We know $s_1Y_j\in \Mat_m(\mb{Z})$ for all $j=1,\dots r$. 
By Lemma \ref{product of unipotent matrix, clearing denominators} with $\mc{Y}$ being the set of matrices obtained by deleting the last row and column of some $Y_j$, by the formula in the proof of Lemma \ref{product of unipotent matrix, clearing denominators}, and by the induction hypothesis, we have
\begin{align*}
\hgt\left(\prod_{j=1}^rY_j\right)_{1,m}
&\leq s_1^{m-1}\hgt\left( \sum_{i=1}^r s_1^{m-1}(Y_i)_{1,m}+\sum_{k=2}^r\sum_{i=2}^{m-1}s_1^{m-2} \left(\prod_{j=1}^{k-1}Y_j\right)_{1,i}s_1(Y_k)_{i,m}\right)
\\&\leq s_1^{m-1}\left(\sum_{i=1}^r |s_1^{m-1}(Y_i)_{1,m}|+\sum_{k=2}^r\sum_{i=2}^{m-1}\left|s_1^{m-2} \left(\prod_{j=1}^{k-1}Y_j\right)_{1,i}\right||s_1(Y_k)_{i,m}|\right)
\\&\leq O(l_0^m\cdot r)+ \sum_{k=2}^r\sum_{i=2}^{m-1} O(l_0^{2^{m-1}(m-1)}(k-1)^{m-2}\cdot l_0^m)
\\&\leq O(l_0^m\cdot r)+ \sum_{k=2}^r O(l_0^{2^mm}(r-1)^{m-2})
\\&\leq O(l_0^{2^mm} \cdot r^{m-1})
\end{align*}
By induction assumption, the heights of the other entries are less than $\max\{C^\prime, C^{\prime\prime}\}\cdot l_0^{2^{m-1}(m-1)}\cdot r^{m-2}$.
The lemma follows.
\end{proof}

\begin{thm}\label{lower bound of number of vertices in paths}
There are constants $c_7,c_8>0$ such that when $T\gg 0$, the graph $Q_{\log T^{1/\lambda}}$ has at least $c_7T^{c_8}$ vertices of heights at most $T$.
\end{thm}

\begin{proof}
The assumption in case (2) says $\hgt_{\tbW} Z(T)> e^{\lambda T}$ for all $T\gg 0$.  Fix such large $T$, such that $e^{\lambda T}>1=\hgt[id]$. Since $h\in \Phi_{\mb{R}}$, the identity $[id]$ is in $V(Q_T)$ for all $T> 0$.  Then by Lemma \ref{graph connected}, there exist $[\gamma]\in V(Q_T)$ for which $\hgt[\gamma] >e^{\lambda T}$ and a path in the graph $Q_T$ joining $[id]$ and $[\gamma]=:[\gamma_r]$ with intermediate vertices $[\gamma_{1}],\dots, [\gamma_{r-1}]$ of heights less than $e^{\lambda T}$, where $[\gamma_{i}]$ and $[\gamma_{i+1}]$ are adjacent. 
By Lemma \ref{adjacent vertices and heights},
$[\gamma]=\prod_{j=0}^{r-1} [\eta_j][\delta_j]^{\tbW}[\eta_j]^{-1}$
 for some  $\delta_j\in S$ and some $\eta_j\in \tbL(\mb{Z})^+$ such that $\hgt [\eta_j]\leq e^{2k_2c_1T}$. 
The quotient map $\tbP(\mb{Q})\ra (\tbP/\tbN)(\mb{Q})$ is defined by finitely many rational polynomials. 
Let $f\in \mb{Z}^+$ such that $f[\eta]\in (\tbP/\tbN)(\mb{Z})$ for all   $\eta\in \tbL(\mb{Z})^+$.
Let $\mc{S}=\{[\delta]^{\tbW}: \delta\in S\}$, $\mc{L}=\{[\eta_0],\dots , [\eta_{r-1}]\}$, and $l_0= e^{2k_2c_1T}$. By Lemma \ref{unipotent matrix height bound}  with such $\mc{S}, f, \mc{L}$ and $l_0$,  we have
\begin{align*}
e^{\lambda T}<\hgt[\gamma]=
\hgt\prod_{j=0}^{r-1} [\eta_j][\delta_j]^{\tbW}[\eta_j]^{-1}
=O(e^{2k_2c_1T2^mm} r^{m-1}).
\end{align*}
Hence, there are constants $c_5,c_6>0$ ($c_6>0$ because $\lambda>2k_2c_12^mm$ by assumption in Section \ref{Heights}) independent of $T$ such that  $Q_T$ has at least $c_5e^{c_6 T}$ vertices of heights at most $e^{\lambda T}$. The theorem follows.
 \end{proof}

\begin{thm}
Theorem \ref{Ax-Schanuel for variations of mixed Hodge structures} holds in case (2).
\end{thm}

\begin{proof}
By Theorem \ref{lower bound of number of vertices in paths}, when $T\gg 0$,  there are at least $c_7T^{c_8}$ points $[\gamma]$ in $\tbW(\mb{Q})$ of heights at most $T$ such that  for some $\gamma^\prime \in [\gamma]\cap \tbU(\mb{Z})$, 
$$Z\cap  \gamma^\prime\bigcup_{\substack{ \eta\in \tb{L}(\mb{Z})^+,\sigma \in \tb{N}_r(\mb{Z})^+}}  \sigma\eta \Phi_{\mb{R}}\neq \varnothing.$$
Since $Z=r(p_D(U))$ and $r$ is $\tbP(\mb{Z})^+$-equivariant, this condition implies that
$$U\cap (X\times \gamma^\prime \bigcup_{\substack{  \eta\in \tb{L}(\mb{Z})^+,\sigma \in \tb{N}_r(\mb{Z})^+}}   \sigma\eta  \Phi)\neq \varnothing.$$
By Lemma \ref{intersecting fundamental domain, definable quotient}, when $T\gg 0$,  $I_{\tbW}$ contains at least $c_7T^{c_8}$ points $[\gamma]$ in $\tbW(\mb{Q})$ of heights at most $T$ such that $\gamma\in \tbU(\ZZ)$.
By the counting theorem of Pila-Wilkie \cite{PW},  for any positive integer $p$, the set $I_{\tbW}$ contains a semialgebraic curve $C_{\mb{R}}$ containing at least $p$ points of the form $[\gamma]$, where $\gamma\in \tbU(\ZZ)$.
By Lemma \ref{inifinite graph, AS}, the theorem follows.
\end{proof}

\section{Proof of case (3)}\label{Proof of case (3)}
Since $\hgt_{\tbW} Z$ is finite, the set
 $$\mc{T}:=\{[\gamma]: \gamma\in\tbU(\mb{Z}), Z\cap  \gamma \bigcup_{\substack{\sigma\in \tbN_r(\mb{Z})^+}}   \sigma \Phi_{\mb{R}}\neq \varnothing \}$$
 is finite. Write $\mc{T}=\{[\gamma_1],\dots, [\gamma_n]\}$, where $\gamma_i$ satisfy the conditions in $\mc{T}$.
Also, $Z$ is contained in a fiber of $p_{\tbL}$. We have $p_{\tbL}(Z)\in \Phi_{\tbL}$. 
Recall that we let $\tbN(\ZZ)^+:=\tbN(\ZZ)\cap \tbP(\RR)^+$.
\begin{lemma}\label{graph connected, case 3}
We have 
$$Z\subset \bigcup_{i=1}^n \bigcup_{\substack{\sigma\in \tbN(\mb{Z})^+}} \gamma_i\sigma \Phi_{\mb{R}}.$$
\end{lemma}

\begin{proof}
Suppose $z\in Z$. Let $(z_{\tbN_r}, z_{\tbL}, z_{\tbU})$ be the image of $z$ under the isomorphism 
$$r(D)\simeq D_{\tb{N}_r}\times D_{\tbL}\times D_{\tb{U},\mb{R}}.$$
 There exists $ \sigma \in \tbN_r(\mb{Z})^+$ such that $z_{\tbN_r}\in \sigma\Phi_{\tbN_r}$. 
 Recall that $\tbP(\mb{R})^+$ acts on $D_{\tbU,\mb{R}}$ \emph{a priori}.  There exists   $\gamma\in \tbU(\mb{Z})$ such that  $z_{\tbU}\in \gamma\sigma\Phi_{\tbU}$.
For some  $g_1\in \tbN_r(\mb{R})^+$, $g_2\in \tbL(\mb{R})^+$ and $u\in \tbU(\mb{R})$, we have $g_1{h}_{\text{Gr}}\in \Phi_{\tbN_r}$, $g_2{h}_{\text{Gr}}\in \Phi_{\tbL}$, $u{h}_s\in \Phi_{\tbU}$, and 
$$ (z_{\tbN_r}, z_{\tbL},z_{\tb{U}})=(\sigma g_1{h}_{\text{Gr}}, g_2{h}_{\text{Gr}}, \gamma\sigma u{h}_s).$$ Let $g:=g_1g_2$.
There exists $u^\prime\in\tbU(\mb{R})$ such that $g^{-1}{h}_s=u^\prime {h}_s$.
Then since $\tbU(\mb{R})$ acts trivially on the associated graded, we have 
$$ugu^\prime h =j^{-1}((g_1{h}_{\text{Gr}}, g_2{h}_{\text{Gr}}, ugu^\prime {h}_s))\in \Phi_{\mb{R}},$$
and thus 
$$z=j^{-1}((z_{\tbN_r}, z_{\tbL},z_{\tb{U}}))=\gamma\sigma ugu^\prime h\in \gamma\sigma \Phi_{\mb{R}}.$$
The inclusion follows from the definition of $\mc{T}$.
\end{proof}

\begin{thm}
Theorem \ref{Ax-Schanuel for variations of mixed Hodge structures} holds in case (3). 
\end{thm}

\begin{proof}
Firstly, note that to prove Theorem \ref{Ax-Schanuel for variations of mixed Hodge structures}, it suffices to prove it for irreducible components of $V\cap W$, so we can assume $U$ is an irreducible component of $V\cap W$.
By Lemma \ref{graph connected, case 3}, by the definition of $Z$ and that $r$ is $\tbP(\mb{Z})^+$-equivariant, 
\begin{align*}
U&=\bigcup_{i=1}^n\bigcup_{\sigma\in \tbN(\mb{Z})^+}U\cap \left(X\times \gamma_i\sigma \Phi\right)
=\bigcup_{i=1}^n\bigcup_{\sigma\in \tbN(\mb{Z})^+} \gamma_i^{-1} \sigma^{-1} U\cap \left(X\times\Phi\right);
\end{align*}
here $\sigma^{-1}$ and $\gamma_i^{-1}$ can be switched because  $\tbN(\mb{Z})^+:=\tbN(\ZZ)\cap \tbP(\RR)^+$ is normal in $\tbP(\mb{Z})^+$. Since $V$ is algebraic and invariant under $Stab(V)$, and since $\tbN_{\mb{C}}$ is the identity component of the $\mb{C}$-Zariski closure of $Stab(V)$, $V$ is invariant under  $\tbN_{\mb{C}}$. We know
$\gamma_i^{-1}\sigma^{-1}U\cap (X\times \Phi)$ is a finite union of components of 
$$\gamma_i^{-1}\sigma^{-1}(V\cap W)\cap (X\times \Phi)=\gamma_i^{-1}V\cap W\cap (X\times \Phi).$$
Since $\gamma_i^{-1}V\cap W\cap (X\times \Phi)$ has only finitely many components, $\gamma_i^{-1}\sigma^{-1}U\cap (X\times \Phi)$ are equal to finitely many possible sets as $\sigma$ varies, each of which is definable. 
The sets $U$ and $p_X(U)$ are then definable.

For any $(x,d)\in W\cap \bigcup_{i=1}^n (X\times \gamma_i\Phi)$, we know $d\in \gamma_i\Phi$ for some $i=1,\dots, n$, and $\varphi(x)=\pi(d)$.
It follows that for any fixed $x\in X$, if $(x,d)\in W$, then the possible values for $d$ are $\tbP(\ZZ)^+$-translates of each other, thus $d$ can take at most finitely many values (denoted by $d_{x,1},\dots, d_{x,k(x)}$) because $d$ is in finitely many translates of the fundamental domain $\Phi$, which overlaps with finitely many translates of it.
For each $x\in X$ and a choice of $d_{x,j}$ above for some $j=1,\dots, k(x)$, we can choose a local lifting of the period mapping $\varphi$ on some open subset $B_{x,j}$ of $X$ containing $x$ that maps $x$ to $d_{x,j}$. 
We can choose $B_{x,j}$ small enough such that the image $I_{x,j}$ of the local lifting $\varphi_{x,j}$ is bounded and contained in $\gamma_i\Phi$.
Let $B_x$ be the intersection of all the (finitely many) $B_{x,j}$.
For any $x'\in B_x$, we know $d_{x',1},\dots, d_{x',k(x')}\in I_{x,1}\cup\dots,\cup I_{x,k(x)}$.
Suppose $K$ is a compact subset of $X$. 
Choose $x_1,\dots, x_{\ell}\in X$ such that $K\subset B_{x_1}\cup \dots \cup B_{x_\ell}$.
We have 
$$p_X^{-1}(K)\cap  W\cap \bigcup_{i=1}^n (X\times \gamma_i\Phi)\subset K\times \bigcup_{d=1}^{\ell} \bigcup_{j=1}^{k(x_d)} I_{x_d, j}.$$
It follows that $p_X|_U$ is proper because we have proved earlier that 
$$U\subset \bigcup_{i=1}^n (X\times \gamma_i\Phi).$$
The set $p_X(U)$ is then complex analytically constructible in $X$ by Chevalley-Remmert Theorem \cite[p. 291]{L}.
By definable Chow theorem \cite{PS} (see also  \cite{MPT}), $p_X(U)$ is algebraically constructible. 
By Lemma \ref{algebraic projection}, Theorem \ref{Ax-Schanuel for variations of mixed Hodge structures} holds.
\end{proof} 

\appendix
\section{Weak Mumford-Tate domains}\label{Appendix}
Let $\wc{\mc{M}}$ be the projective space defined in  \cite[\S 3.5]{BBKT} parametrizing decreasing filtrations with fixed graded-polarization and Hodge numbers. The graded-polarization and Hodge numbers are chosen to be the same as that of the mixed Hodge structures our VMHS (in Section \ref{Statement of results}) is parametrizing. 
Let $h$ be any mixed Hodge structure in $\mc{M}$.
Let $\textbf{M}$ be a normal algebraic $\mb{Q}$-subgroup of the Mumford-Tate group $\mcMT_h$ of $h$. Let $\tb{M}_u$ be its unipotent radical.  
Let $\tbM(\RR)^+$ be the identity component of $\tbM(\RR)$. 
Let $D(\textbf{M})$ be the $\textbf{M}(\mb{R})^+\tb{M}_u(\mb{C})$-orbit of $h$ in $\wc{\mc{M}}$. 
Let $\wc{D}(\tbM)$ be the $\tbM(\CC)$-orbit of $h$  in $\wc{\mc{M}}$.
It is a complex algebraic set.

\begin{thm}\label{algebraicity of weak MT domain, dual}
The weak Mumford-Tate domain $D(\tbM)$ is  open in $\wc{D}(\tbM)$ in the Archimedean topology, so it inherits a complex analytic structure from $\wc{D}(\tbM)$.
\end{thm}

\begin{proof} 
Let $P^\mc{M}$ be the real algebraic group called $\tbG$ in  \cite[\S 3.5]{BBKT}, whose complex points acts transitively on $\wc{\mc{M}}$.
Since $\tbM$ is normal in $\mcMT_{h}$, the adjoint action of $\mcMT_{h}$ stabilizes $\mf{m}:=\Lie \tbM$. Hence,  $\mf{m}$ is equipped with a mixed Hodge structure. 
Let  $\mf{p}^{\mc{M}}_{\CC}=\bigoplus_{r,s} \mf{p}^{\mc{M}, r,s}$ and $\mf{m}_\mb{C}=\bigoplus_{r,s} \mf{m}^{r,s}$ be the Deligne bigradings \cite{D} of the mixed Hodge structures on $\mf{p}^{\mc{M}}:=\Lie P^{\mc{M}}$ and $\mf{m}$ respectively.  
By functoriality, $\mf{m}^{r,s}=\mf{m}_{\CC}\cap \mf{p}^{\mc{M},r,s}$.
Let $\mf{b}^{\mc{M}}$ be the Lie algebra of the $P^{\mc{M}}(\CC)$-stabilizer $B^{\mc{M}}$ of $h$. 
By (21) of \cite{PP},  
$$\mf{b}^{\mc{M}}=F^0\mf{p}^{\mc{M}}_{\CC}=\bigoplus_{r\geq 0;s} \mf{p}^{\mc{M},r,s}.$$
Let $\mf{b}$ be the Lie algebra of the $\tbM(\CC)$-stabilizer $B$ of $h$. Then 
$$\mf{b}=\mf{m}_{\CC}\cap \mf{b}^{\mc{M}}=\mf{m}_{\CC}\cap \bigoplus_{r\geq 0;s} \mf{p}^{\mc{M},r,s}=\bigoplus_{r\geq 0;s}\mf{m}^{r,s}.$$
Since $\bigoplus_{r+s\leq -1}\mf{m}^{r,s}$ is a nilpotent ideal of $\mf{m}_{\CC}$, we know 
$\mf{m}_{\CC}=\mf{m}_{\CC, u}+F^0\mf{m}_{\CC}+ \ol{F^0\mf{m}_{\CC}}$, where $\mf{m}_{\CC, u}$ is the Lie algebra of the unipotent radical of $\tbM(\CC)$.
For any $X\in \ol{F^0\mf{m}_{\CC}}$, we have  $X=(X+\ol{X})-\ol{X}\in \mf{m}_{\RR}+F^0\mf{m}_{\CC}$.
Hence, $\mf{m}_{\CC}\subset\mf{m}_{\RR}+\mf{m}_{\CC,u}+F^0\mf{m}_{\CC}$.
Therefore, the canonical map 
$$(\mf{m}_{\RR}+\mf{m}_{\CC,u})/ ((\mf{m}_{\RR}+\mf{m}_{\CC,u})\cap \mf{b})\ra \mf{m}_{\CC}/\mf{b}$$
is surjective. The canonical map $D(\tbM)\ra \wc{D}(\tbM)$ is then a submersion, and thus it is an open embedding. 
\end{proof}

 For simplicity, write $\mcMT:=\mcMT_{h}$.
The mixed Hodge structure $h$ defines a representation $\rho_{\CC}:\tb{S}_\mb{C}\ra \textbf{GL}(\mc{H}_{\mb{C},\eta})$ \cite[p. 7]{K}. Let $\mc{X}_{\mcMT}$ be the $\mcMT(\RR)\mcMT_u(\CC)$-conjugacy class of $\rho_{\CC}$.
 Let $D_{\mcMT}^+$ be the connected mixed Mumford-Tate domain, i.e. the $\mcMT(\RR)^+\mcMT_u(\CC)$-orbit of $h$.
The tuple $(\mcMT, \mc{X}_{\mcMT}, D_{\mcMT}^+)$ is a connected mixed Hodge datum \cite[p. 10]{K}.

Let $\textbf{M}$ be a normal algebraic $\mb{Q}$-subgroup of $\mcMT$.
Composing $\rho_{\CC}$ with the quotient map $\mcMT(\CC)\ra (\mcMT/\tbM)(\CC)$ gives a representation 
$$\ol{\rho}_{\CC}:\tb{S}_{\CC}\ra (\mcMT/\tbM)(\CC).$$ 
Let $\mc{X}_{\mcMT/\tbM}$ be the  $(\mcMT/\tbM)(\RR)(\mcMT/\tbM)_u(\CC)$-conjugacy class of $\ol{\rho}_{\CC}$ in the set $\Hom(\tb{S}_{\CC}, (\mc{MT}/\tbM)(\CC))$. Fix an embedding of $\mcMT/\tbM$ into the automorphism group of some vector space.

\begin{lemma}\label{quotient of mixed hodge structure}
The representation $\ol{\rho}_{\CC}$ satisfies (1), (2), and (3) of \cite[Prop. 2.3]{K}.
\end{lemma}

\begin{proof}
 By \cite[Prop. 2.3]{K}, $\rho_{\CC}$ satisfies (1), (2), and (3).
The representation $\ol{\rho}_{\CC}$ satisfies (1) and (2)  because we have the $\QQ$-morphism $\mcMT_u\ra (\mcMT/\tbM)_u$. Let $\Ad_1: \mcMT\ra \tbGL(\mf{mt})$ and $\Ad_2:\mcMT/\tbM\ra \tbGL(\mf{mt}/\mf{m})$ be the adjoint representations. By the functoriality of the adjoint representation, we have the following commutative diagram
\begin{center}
\begin{tikzcd}
\tb{S}_{\CC}\arrow[r]& \mcMT(\CC) \arrow[d, "\Ad_1"]\arrow[r] & (\mcMT/\tbM)(\CC)\arrow[d, "\Ad_2"]
\\&  \Ad(\mc{MT}(\CC))\arrow[r]& \Ad((\mcMT/\tbM)(\CC)).
\end{tikzcd}
\end{center}
Hence, the differential $\mf{mt}\ra \mf{mt}/\mf{m}$ preserves the gradings, so  $\ol{\rho}_{\CC}$ satisfies (3).
\end{proof}
 
Fix Levi subgroups $\tbM_r$ and $\mcMT_r$ for $\tbM$ and $\mcMT$ respectively such that $\tbM_r\subset \mcMT_r$. By \cite[Corollary 14.11]{Bor}, $\mcMT_u\tbM/\tbM$ is the unipotent radical of $\mcMT/\tbM$. Let $(\mcMT/\tbM)_r:=\mcMT_r\tbM/\tbM$, which is a Levi subgroup of $\mcMT/\tbM$. Since $\tbM$ is normal in $\mcMT$, we have $\tbM_r=\tbM\cap \mcMT_r$ and $\tbM_u=\tbM\cap \mcMT_u$  by \cite[Prop. 2.13]{G2}. Hence, $(\mcMT/\tbM)_u\simeq \mcMT_u/\tbM_u$  and $(\mcMT/\tbM)_r\simeq \mcMT_r/\tbM_r$.

By Lemma \ref{quotient of mixed hodge structure} and \cite[Prop. 3.1]{K},  there is a complex manifold $D_{\mcMT/\tbM}$ attached to $\mc{X}_{\mcMT/\tbM}$. We also have a connected mixed Hodge data morphism 
$$(\mcMT, \mc{X}_{\mcMT}, D_{\mcMT}^+)\ra (\mcMT/\tbM,\mc{X}_{\mcMT/\tbM},D_{\mcMT/\tbM}^+),$$
where the map $D_{\mcMT}^+\ra D_{\mcMT/\tbM}^+$ is given by $\gamma_r\gamma_u \cdot h\mapsto (\gamma_r \tbM_r(\RR)^+)(\gamma_u \tbM_u(\CC))\cdot \ol{h}$ for any $\gamma_r\in \mcMT_r(\RR)^+$ and $\gamma_u\in \mcMT_u(\CC)$, and $\ol{h}$ is the mixed Hodge structure attached to $\ol{\rho}_{\CC}$.

Let $h_0$, $\rho_0$, $\tbP$, $\tbU$ and $D$ be defined as in Section  \ref{Statement of results}.
Let $h$ be any point in $D_{\mc{MT}}^+$.
 By Andr\'{e} \cite[Proof of Theorem 1]{A}, $\tbP$ is normal in $\mcMT$.
Let $f$ be the  morphism from the connected mixed Hodge datum $(\mcMT, \mc{X}_{\mcMT}, D_{\mcMT}^+)$ to the connected mixed Hodge datum  $(\mcMT/\tbP, \mc{X}_{\mcMT/\tbP}, D_{\mcMT/\tbP}^+)$ as above.
Let $\mf{mt}_\mb{C}=\bigoplus_{p,q} \mf{mt}^{p,q}$ be the Deligne bigrading \cite{D} of the mixed Hodge structure on the Lie algebra $\mf{mt}$ of $\mcMT$. Let $\mf{b}$ be the Lie algebra of the stabilizer  $B$  in $\mcMT(\mb{C})$ of $h$. By Remark 2.4 and (21) of \cite{PP},  $\mf{b}$ can be identified with $\bigoplus_{p\geq 0; q} \mf{mt}^{p,q}$. Let $\mf{v}$ be the Lie algebra of $V=\mcMT(\mb{R})^+\mcMT_u(\mb{C})\cap B$.  By the definition of mixed Hodge datum, $\mf{mt}_{u,\mb{C}}=\bigoplus_{p+q\leq -1}\mf{mt}^{p,q}$. We have  
\begin{align*}
\mf{v}
&=(\mf{mt}_{\mb{R}}+\mf{mt}_{u,\mb{C}})\cap \mf{b}
\\&=\left(\left(\bigoplus_{p+q= 0} \mf{mt}^{p,q}\right)_{\RR}\oplus \left(\bigoplus_{p+q\leq -1}\mf{mt}^{p,q}\right)\right)\cap \bigoplus_{p\geq 0; q} \mf{mt}^{p,q}
\\&=\mf{mt}^{0,0}_\mb{R}\oplus \left(\bigoplus_{p+q\leq -1;p\geq 0} \mf{mt}^{p,q}\right).
\end{align*}
Note that we have $p+q\leq 0$ in the above calculation because the Lie algebra action preserves weight filtration.
From this expression of $\mf{v}$, we have an identification
$$\alpha: T_{h}D_{\mcMT}^+ \simeq  \left(\bigoplus_{p\neq 0} \mf{mt}^{p,-p}\right)_\mb{R}\oplus \left(\bigoplus_{p+q\leq -1;p<0} \mf{mt}^{p,q}\right)=: (\mf{mt}_{\mb{R}}+\mf{mt}_{u,\mb{C}})^-.$$
Let $\beta$ be the projection 
$$\mf{mt}_{\mb{R}}+\mf{mt}_{u,\mb{C}}=\left(\bigoplus_{p+q=0} \mf{mt}^{p,q}\right)_\mb{R}\oplus \left(\bigoplus_{p+q\leq -1}\mf{mt}^{p,q}\right)\ra (\mf{mt}_{\mb{R}}+\mf{mt}_{u,\mb{C}})^-.$$
Denote the kernel of $\beta$ by $(\mf{mt}_{\mb{R}}+\mf{mt}_{u,\mb{C}})^+$.
 Let $\mf{p}$ and $\mf{u}$ be the Lie algebras of $\tbP$ and $\tbU$ respectively.
Replacing $\mf{mt}$ by $\mf{mt}/\mf{p}$, we also have maps $\gamma$ and $\delta$, akin to $\alpha$ and $\beta$ respectively.

\begin{lemma}\label{D and fiber of quotient}
Suppose $D_0$ is an irreducible subvariety of $D_{\mc{MT}}^+$ invariant under $\tbP(\RR)^+\tbU(\CC)$ and suppose its image under $f$ is a point. Then $\tbP(\RR)^+\tbU(\CC)$ acts transitively on $D_0$, and $D_0\subset D$. In particular, if $D_0$ is the connected component  of $f^{-1}(f(h_0))$ that contains some $h_0\in D$, then  $D_0=D$. 
\end{lemma}

\begin{proof}
Let $h\in D_0$.
Define a map $g_1:\mcMT(\mb{R})^+\mcMT_u(\mb{C})\ra D_{\mcMT}^+$ by $m\mapsto m\cdot h$ for any $m\in \mcMT(\mb{R})^+\mcMT_u(\mb{C})$. Similarly, we have a map $g_2: \tbP(\mb{R})^+\tbU(\mb{C})\ra D_0$, and also a map $$g_3:(\mcMT/\tb{P})(\mb{R})^+(\mcMT/\tb{P})_u(\mb{C})\ra D_{\mcMT/\tbP}^+.$$ The differentials $dg_1$ and $dg_3$ of $g_1$ and $g_3$ are $\alpha^{-1}\circ \beta$ and $\gamma^{-1}\circ \delta$ respectively. 
Let $\tbP_r$ be a fixed Levi subgroup of $\tbP$ that is contained in $\mcMT_r$. Let $(\mcMT/\tbP)_r:=\mcMT_r\tbP/\tbP$, which is a Levi subgroup of $\mcMT/\tbP$.
Note that $\mf{p}_{r,\mb{R}}\oplus\mf{u}_{\mb{C}}=\mf{p}_{\RR}+\mf{u}_{\CC}$, $ \mf{mt}_{r,\mb{R}}\oplus\mf{mt}_{u,\mb{C}}=\mf{mt}_{\RR}+\mf{mt}_{u,\CC}$ and $(\mf{mt}/\mf{p})_{r,\mb{R}}\oplus(\mf{mt}/\mf{p})_{u,\mb{C}}=(\mf{mt}/\mf{p})_{\RR}+(\mf{mt}/\mf{p})_{u,\CC}$.
We have the following commutative diagram of differentials
\begin{center}
\begin{tikzcd}
T_{h_0}D_0\arrow[r, hook] & T_{h_0}D_{\mcMT}^+\arrow[r] & T_{f(h_0)}D_{\mcMT/\tbP}^+
\\\mf{p}_{r,\mb{R}}\oplus\mf{u}_{\mb{C}}\arrow[r, hook] \arrow[u, "dg_2"]& \mf{mt}_{r,\mb{R}}\oplus\mf{mt}_{u,\mb{C}} \arrow[r, "q"]\arrow[u, "dg_1"]& (\mf{mt}/\mf{p})_{r,\mb{R}}\oplus(\mf{mt}/\mf{p})_{u,\mb{C}}\arrow[u, "dg_3"],
\end{tikzcd}
\end{center}
where the composition of the upper horizontal maps is zero, and where $q$ is the componentwise quotient by $\mf{p}_{r,\RR}$ and $\mf{p}_{u,\CC}$. Let $v\in T_{h_0}D_0$. Since $\mcMT(\mb{R})^+\mcMT_u(\mb{C})$ acts transitively on $D_{\mcMT}^+$, there exists $w\in  \mf{mt}_{\mb{R}}+\mf{mt}_{u,\mb{C}}$ such that $(dg_1)(w)=v$. Write $w=w^-+w^+$, where $w^-\in (\mf{mt}_{\mb{R}}+\mf{mt}_{u,\mb{C}})^-$ and $w^+\in (\mf{mt}_{\mb{R}}+\mf{mt}_{u,\mb{C}})^+$. Then $(dg_1)(w^-)=(dg_1)(w)-(dg_1)(w^+)=v$.   
By commutativity, we then know $q(w^-)$ is in the kernel of $dg_3=\gamma^{-1}\circ\delta$, so $q(w^-)\in ((\mf{mt}/\mf{p})_{\mb{R}}+(\mf{mt}/\mf{p})_{u,\mb{C}})^+$. Moreover, the quotient morphism $f$ of connected mixed Hodge data induces a morphism $df: \mf{mt}\ra \mf{mt}/\mf{p}$ of mixed Hodge structures, so $q(w^-)\in ((\mf{mt}/\mf{p})_{\mb{R}}+(\mf{mt}/\mf{p})_{u,\mb{C}})^-$. 
Hence, $q(w^-)=0$. Since $q$ is the componentwise quotient by $\mf{p}_{r,\RR}$ and $\mf{p}_{u,\CC}$, and since $\mf{p}_{u,\CC}=\mf{u}_{\CC}$, we thus have $w^-\in \mf{p}_{r,\mb{R}}+\mf{u}_\mb{C}$. Then since $dg_2$ is the restriction of $dg_1$, the differential $dg_2$ is surjective.  Since $h\in D_0$ is arbitrary, we know $g_2$ is a submersion, and thus an open map.  Again since $h\in D_0$ is arbitrary, every $\tbP(\mb{R})^+\tbU(\mb{C})$-orbit in $D_0$ is open. By the connectedness of $D_0$, the group $\tbP(\mb{R})^+\tbU(\mb{C})$ acts transitively on $D_0$. Therefore, $D_0\subset D$. Since $f(D)=\{f(h_0)\}$, if $D_0$ is the connected component of $f^{-1}(f(h_0))$ that contains  $h_0$, we then have $D\subset D_0$, so $D=D_0$.
\end{proof}

\begin{lemma}\label{Galois descent, Zariski closure}
Let $K$ be a subfield of $\CC$. Let $E$ be a subset of the set $\mb{A}^n_K(K)$ of $K$-points of the $n$-dimensional affine space. Let $M$ be the smallest closed algebraic $K$-subvariety of $\mb{A}^n_K$ such that $E\subset M(K)$. The complex variety $M_{\CC}$ is the smallest closed algebraic $\CC$-subvariety of $\mb{A}^n_{\CC}$ such that $E\subset M(\CC)$.
\end{lemma}

\begin{proof}
Suppose $M'_{\CC}$ is a closed algebraic $\CC$-subvariety of $\mb{A}^n_{\CC}$ such that $E\subset M'_{\CC}(\CC)$. The complex algebraic variety $\bigcap_{\sigma\in \Gal(\CC/K)} \sigma M'_{\CC}$ is stable under the Galois action by $\Gal(\CC/K)$. Since $K$ is perfect, the variety $\bigcap_{\sigma\in \Gal(\CC/K)} \sigma M'_{\CC}$ has a closed affine model $L$ over $K$  by Galois descent \cite[Prop. 16.1, 16.8]{Mil}. Since $E\subset \mb{A}^n_K(K)$, we know $\sigma E=E$ for any $\sigma\in \Gal(\CC/K)$. Hence, $E\subset L(K)$, and thus $M\subset L$. Therefore, $M_{\CC}\subset L_{\CC}\subset M'_{\CC}$, as desired.
\end{proof}

Recall $X$ and $D$ in Section \ref{Statement of results}. Let $\wt{X}$ be the universal cover of $X$.

\begin{lemma}\label{weak period map}
The domain $D$ contains the image $\Pi$ of the period mapping $\wt{X}\ra  D_{\mcMT}^+$. Let $\Pi^{Zar}$  be the Zariski closure of $\Pi$ in the $\tbP(\CC)$-orbit $\wc{D}$ of $h$. The domain $D$ contains an open subset of $\Pi^{\Zar}$. Moreover, $\Pi^{\Zar}$ contains $D$.
\end{lemma}

\begin{proof}
Since $\tbP(\mb{Q})\cap \Gamma$ is of finite index in $\Gamma$, the composition $\wt{X}\ra D_{\mcMT}^+\ra D_{\mcMT/\tbP}^+$ descends to a period mapping on $X$ whose associated GPVMHS has finite monodromy.
Replacing $X$ by a finite \'{e}tale covering if necessary, we have a GPVMHS with trivial monodromy. By rigidity \cite[Theorem 7.12]{BZ}, this GPVMHS and its associated period mapping are constant. Hence, the lifting $\widetilde{X}\ra D_{\mcMT}^+\ra  D_{\mcMT/\tbP}^+$ is constant, with value $f(h_0)$. Hence,  the image $\Pi$ of $\wt{X}$ in $D_{\mcMT}^+$ lies in the connected component  of $f^{-1}(f(h_0))$ that contains  $h_0$, and thus lies in $D$ by Lemma \ref{D and fiber of quotient}. 
By Lemma \ref{algebraicity of weak MT domain, dual}, the domain $D$ is open in $\wc{D}$. Since $\wc{D}$ is algebraic, it contains $\Pi^{\Zar}$. Thus, $D\cap \Pi^{\Zar}$ is open in $\Pi^{\Zar}$. The image $\Pi$ is invariant under the monodromy group of the variation. Since $\tbP$ is defined as the identity component of the $\QQ$-Zariski closure of this monodromy group, $\Pi^{\Zar}$ is invariant under $\tbP$, so $\Pi^{\Zar}$ contains $D$.  
\end{proof}

For an irreducible analytic set $E\subset X\times \widecheck{D}$, let $\tbP_E$ be the connected algebraic  monodromy group of the GPVMHS restricted to the smooth locus of $p_X(E)^{\Zar}$. Let $\varphi: X\ra \tbP(\ZZ)^+\bs D$ be the period mapping. Let $\pi: D\ra \tbP(\ZZ)^+\bs D$ be the projection. 
Let $E^{ws}$ be the irreducible component of $\varphi^{-1}\pi(D(\tbP_E))$ containing $p_X(E)$.

\begin{lemma}\label{monodromy and smallest weakly special}
The set $E^{ws}$ is the smallest weakly special subvariety containing $p_X(E)$.
\end{lemma}

\begin{proof}
Suppose $E_1$ is a weakly special subvariety containing $p_X(E)$. Then $E_1$ is an irreducible component of $\varphi^{-1}\pi(D_1)$, where $D_1$ is the $\tbM_1(\RR)^+\tbM_{1,u}(\CC)$-orbit of some mixed Hodge structure $h_1$ in $D$, for some normal algebraic $\QQ$-subgroup $\tbM_1$ of $\mc{MT}_{h_1}$. Since $E_1$ is algebraic \cite[Cor. 6.7]{BBKT}, $E_1$ contains $p_X(E)^{\Zar}$, so $\varphi(p_X(E)^{\Zar})$ is contained in $\pi(D_1)$. Hence, the image $\Pi_E$ of the lifting of the period mapping to the universal cover of the smooth locus $(p_X(E)^{\Zar})^{sm}$ of $p_X(E)^{\Zar}$ is contained in $\Gamma\cdot D_1$. Since this universal cover is irreducible, this image is contained in $D_1$, so $\Pi_E^{\Zar}\subset \wc{D}_1$.  
Hence, by Lemma \ref{weak period map}, $D(\tbP_E)\subset \wc{D}_1$.
By Lemma \ref{algebraicity of weak MT domain, dual}, $D_1$ is open in $\wc{D}_1$. 
The domain $D_1$ contains $D_1\cap D(\tbP_E)$, which is now open in $D(\tbP_E)$. Therefore, $E^{ws}$ is contained in $E_1$.
\end{proof}

For every closed algebraic variety $Z$ of $X$, we denote its algebraic monodromy group (i.e. the identity component of the $\QQ$-Zariski closure of the monodromy group of the GPVMHS restricted to the smooth locus $Z^{sm}$ of $Z$) by $\tbP_Z$.

\begin{lemma}\label{factor through domain}
Let $Z$ be a closed algebraic subvariety of $X$. If $X=\varphi^{-1}\pi (D(\tbP_Z))$, then $\tbP:=\tbP_X=\tbP_Z$.
\end{lemma}

\begin{proof}
If $X=\varphi^{-1}\pi (D(\tbP_Z))$, 
then $D(\tbP_Z)$ contains the image of the lifted period mapping on the universal cover of $X$ (as in the proof of Lemma \ref{weak period map}).
In particular, $\Gamma \cdot h\subset D(\tbP_Z)$ for some $h\in D(\tbP_Z)$.

Let $\varphi: X\ra \Gamma\bs D_{\mcMT}$ be the period mapping. 
Let $D_{\mcMT}\ra \Gamma\bs D_{\mcMT}$ be the projection.
The generic Mumford-Tate group $\mcMT_{Z^{sm}}$ of $Z^{sm}$ is $\mcMT$. 
Otherwise, $\varphi^{-1}\pi_{\mcMT}(D_{\mcMT_{Z^{sm}}})\subsetneq X$.
This contradicts that $X=\varphi^{-1}\pi (D(\tbP_Z))$ because $\varphi^{-1}\pi (D(\tbP_Z))\subset \varphi^{-1}_{\mcMT}\pi_{\mcMT}(D_{\mcMT_{Z^{sm}}})$.
By Andr\'{e} \cite[Proof of Theorem 1]{A}, $\tbP_Z$ is normal in 
$\mcMT:=\mcMT_{Z^{sm}}$.
The quotient map $\mcMT \ra \mcMT/\tbP_Z$ induces a morphism 
$$(\mcMT, D_{\mcMT})\ra (\mcMT/\tbP_Z, D_{\mcMT/\tbP_Z})$$
of mixed Hodge data.
We then know that the image of $\Gamma$ in $\mcMT(\QQ)/\tbP_Z(\QQ)$ is contained in the stabilizer of a point in $D_{\mcMT/\tbP_Z}$. 
Since this stabilizer is compact, the image of $\Gamma$ in $\mcMT(\QQ)/\tbP_Z(\QQ)$ is finite.
Now $\Gamma\cap \tbP_Z(\QQ)$ is of finite index in $\Gamma$, so the $\QQ$-Zariski closure of $\Gamma$ is contained in $\tbP_Z(\QQ)$, and thus $\tbP=\tbP_Z$.
\end{proof}

\begin{thm}\label{Andre-Deligne: monodromy and proper weakly special}
Let $Z$ be a closed algebraic subvariety of $X$. The algebraic monodromy group $\tbP_Z$ is equal to $\tbP_X$ if and only if $Z$ is not contained in any proper weakly special subvariety of $X$.
\end{thm}

\begin{proof}
If $\tbP_Z=\tbP$, then $X$ is the smallest weakly special subvariety containing $Z$ by Lemma \ref{monodromy and smallest weakly special}. 
On the other hand, suppose $Z$ is not contained in a proper weakly special subvariety of $X$. 
By Lemma \ref{monodromy and smallest weakly special}, $X=\varphi^{-1}\pi (D(\tbP_Z))$, 
so $\tbP_X=\tbP_Z$ by Lemma \ref{factor through domain}.
\end{proof}

\end{document}